\documentclass[11pt,fleqn]{article}

\usepackage{amsfonts}
\usepackage{amsmath}
\usepackage{amssymb}
\usepackage{amsthm}
\usepackage{bbm}
\usepackage{bm}
\usepackage{booktabs}
\usepackage{cases}
\usepackage{cite}
\usepackage{graphicx}
\usepackage{indentfirst}
\usepackage{longtable}
\usepackage{makecell}
\usepackage{mathrsfs}
\usepackage{titlesec}
\usepackage{rotating}
\usepackage[T1]{fontenc}
\usepackage[centerlast]{caption}

\allowdisplaybreaks

\numberwithin{equation}{section}

\newtheorem{theorem}{Theorem}[section]
\newtheorem{lemma}[theorem]{Lemma}

\newtheorem{corollary}[theorem]{Corollary}

\theoremstyle{definition}

\newtheorem{example}{Example}[section]

\titleformat{\section}
{\normalfont\normalsize\bf}{\thesection.}{6pt}{}
\titleformat{\subsection}
{\normalfont\normalsize}{\thesubsection.}{6pt}{}
\titleformat{\subsubsection}
{\normalfont\normalsize}{\thesubsubsection.}{6pt}{}

\newcommand{\citu}[2]{\!\!\cite[#2]{#1}}

\newcommand{\bibb}[1]{\left\{#1\right\}}

\newcommand{\smbb}[1]{\left(#1\right)}

\newcommand{\binq}[2]{\genfrac{[}{]}{0mm}{0}{#1}{#2}}
\newcommand{\tbnq}[2]{\genfrac{[}{]}{0mm}{1}{#1}{#2}}

\newcommand{\al}{\alpha}
\newcommand{\be}{\beta}
\newcommand{\ga}{\gamma}

\newcommand{\la}{\lambda}
\newcommand{\si}{\sigma}
\newcommand{\vep}{\varepsilon}
\newcommand{\ze}{\zeta}

\newcommand{\ud}{\mathrm{d}}
\newcommand{\uud}{\,\mathrm{d}}

\newcommand{\uue}{\,\mathrm{e}}

\newcommand{\Lii}{\,\mathrm{Li}}

\newcommand{\ol}{\overline}

\newcommand{\bms}{\bm{S}}
\newcommand{\bmt}{\bm{T}}
\newcommand{\bmu}{\bm{U}}

\DeclareMathOperator*{\cat}{\mathbf{Cat}}

\makeatletter
\newcommand{\tmod}[1]{{\@displayfalse\pmod{#1}}}
\makeatother

\makeatletter

\newdimen\bibspace
\renewenvironment{thebibliography}[1]{%
 \section*{\refname %or \bibname if you use ``book'' as the documentclass
       \@mkboth{\MakeUppercase\refname}{\MakeUppercase\refname}}%
     \list{\@biblabel{\@arabic\c@enumiv}}%
          {\settowidth\labelwidth{\@biblabel{#1}}%
           \leftmargin\labelwidth
           \advance\leftmargin\labelsep
           \itemsep\bibspace
           \parsep\z@skip     %
           \@openbib@code
           \usecounter{enumiv}%
           \let\p@enumiv\@empty
           \renewcommand\theenumiv{\@arabic\c@enumiv}}%
     \sloppy\clubpenalty4000\widowpenalty4000%
     \sfcode`\.\@m}
    {\def\@noitemerr
      {\@latex@warning{Empty `thebibliography' environment}}%
     \endlist}

\makeatother

\begin{document}

\title{\bf\boldmath{Alternating multiple zeta values, and explicit formulas of some Euler-Ap\'{e}ry-type series}}
\author{
{
Weiping Wang$^{a,\,}$\thanks{E-mail\,:
wpingwang@yahoo.com, wpingwang@zstu.edu.cn (Weiping Wang).}
\quad
Ce Xu$^{b,\,c,\,}$\thanks{Corresponding author. E-mail\,:
19020170155420@stu.xmu.edu.cn, 9ma18001g@math.kyushu-u.ac.jp (Ce Xu).}
}\\[1mm]
\small a. School of Science, Zhejiang Sci-Tech University, Hangzhou 310018, P.R. China\\
\small b. Multiple Zeta Research Center, Kyushu University,
    Motooka, Nishi-ku, Fukuoka 819-0389, Japan\\
\small c. School of Mathematical Sciences, Xiamen University, Xiamen 361005, P.R. China}

\date{}
\maketitle

\vspace{-0.5cm}
\begin{center}
\parbox{6.3in}{\small{\bf Abstract}\vspace{3pt}

\hspace{3.5ex}In this paper, we study some Euler-Ap\'{e}ry-type series which involve central binomial coefficients and (generalized) harmonic numbers. In particular, we establish elegant explicit formulas of some series by iterated integrals and alternating multiple zeta values. Based on these formulas, we further show that some other series are reducible to $\ln(2)$, zeta values, and alternating multiple zeta values by considering the contour integrals related to gamma functions, polygamma functions and trigonometric functions. The evaluations of a large number of special Euler-Ap\'{e}ry-type series are presented as examples.
}

\vspace{6pt}
\parbox{6.3in}{\small{\emph{AMS classification\,:}}\,\,
65B10; 11B65; 11M32}

\vspace{3pt}
\parbox{6.3in}{\small{\emph{Keywords\,:}}\,\,
Alternating multiple zeta values; Euler-Ap\'{e}ry-type series; Central binomial coefficients; Harmonic numbers}
\end{center}

%%%%%%%%%%%%%%%%%%%%%%%%%%%%%%%%%%%%%%%%%%%%%%%%%%%%%%%%%%%%%%%%%%%
%%%%%%%%%%%%%%%%%%%%%%%%%%%%%%%%%%%%%%%%%%%%%%%%%%%%%%%%%%%%%%%%%%%
%%%%%%%%%%%%%%%%%%%%%%%%%%%%%%%%%%%%%%%%%%%%%%%%%%%%%%%%%%%%%%%%%%%
%%%%%%%%%%%%%%%%%%%%%%%%%%%%%%%%%%%%%%%%%%%%%%%%%%%%%%%%%%%%%%%%%%%
%%%%%%%%%%%%%%%%%%%%%%%%%%%%%%%%%%%%%%%%%%%%%%%%%%%%%%%%%%%%%%%%%%%

\setcounter{tocdepth}{2}
\tableofcontents

%%%%%%%%%%%%%%%%%%%%%%%%%%%%%%%%%%%%%%%%%%%%%%%%%%%%%%
%%%%%%%%%%%%%%%%%%%%%%%%%%%%%%%%%%%%%%%%%%%%%%%%%%%%%%
%%%%%%%%%%%%%%%%%%%%%%%%%%%%%%%%%%%%%%%%%%%%%%%%%%%%%%
%%%%%%%%%%%%%%%%%%%%%%%%%%%%%%%%%%%%%%%%%%%%%%%%%%%%%%
%%%%%%%%%%%%%%%%%%%%%%%%%%%%%%%%%%%%%%%%%%%%%%%%%%%%%%

\section{Introduction}\label{Sec.intro}

Infinite series involving central binomial coefficients and (generalized) harmonic numbers play an important role in many fields, such as analysis of algorithms, combinatorics, number theory and elementary particle physics. Therefore, they have attracted wide attention for a long time.

In 1979, Ap\'{e}ry \cite{Apery79} proved the irrationality of $\ze(3)$ by using the series involving the central binomial coefficients:
\[
\sum_{n=1}^\infty\frac{(-1)^{n-1}}{n^3\binom{2n}{n}}=\frac{2}{5}\ze(3)\,.
\]
He also considered the following analogous one:
\[
\sum_{n=1}^\infty\frac{1}{n^2\binom{2n}{n}}=\frac{1}{3}\ze(2)\,,
\]
which has been known since the nineteenth century. The \emph{Ap\'{e}ry-like series} were later investigated systematically in, for example, \cite{Abling17,AlGra99,BaiBorBr06,BorBK01,ChenKW19,Chu17.HAAS,ChuZh09,DK2004,
HPHP10,JeKaVe03,KWY2007,Leh1985,Lesh81,vdPoor79,Wein04,Zuck85} and the references therein, and a plenty of conjectural Ap\'{e}ry-like series were presented by Sun \cite{Sun14.LCSP,Sun15.NSSS}. It should be noted that, besides the central binomial coefficients, many Ap\'{e}ry-like series also involve the (generalized) harmonic numbers.

The \emph{Euler sums} are a kind of infinite series involving products of (generalized) harmonic numbers, which can be traced back to Goldbach and Euler (see Berndt \cite[p. 253]{Berndt85.1}), and have the general form
\[
\sum_{n=1}^{\infty}\frac{H_n^{(i_1)}H_n^{(i_2)}\cdots H_n^{(i_k)}}{n^q}\,,
\]
where $1\leq i_1\leq i_2\leq\cdots\leq i_k$ and $q\geq 2$. Some well-known works on the Euler sums can be found in, for example, the works due to Bailey et al. \cite{BaiBG94.EM}, Borwein et al. \cite{BorBG95.PEM}, and Flajolet and Salvy \cite{FlSa98}, and some most recent progresses on the Euler sums were given in \cite{WangLyu18.ESSS,Xu17.MZVES,XuWang19.EFES}.

Additionally, for some contributions on various other infinite series involving central binomial coefficients and (generalized) harmonic numbers, the readers are referred to \cite{Boyad12,Boyad17,Camp18,ChenH16,CopCan15,WangChu19} and the references therein.

It is known that some algorithms and program packages such as those developed by Ablinger \cite{Abling17,Abling14} and Weinzierl \cite{Wein04}, can be used to compute \emph{Euler-Ap\'{e}ry-type series} as well as some other infinite binomial and inverse binomial sums. Moreover, explicit formulas for some general Euler-Ap\'{e}ry-type series were established in, for example, \cite{ChenKW19,DK2004,JeKaVe03,KWY2007}. In particular, Kalmykov et al. \cite{KWY2007} showed that all the Euler-Ap\'{e}ry-type series of the form
\[
\sum_{n=1}^\infty
    \frac{z^nH_n^{(i_1)}H_n^{(i_2)}\cdots H_n^{(i_k)}}{n^p\binom{2n}{n}^{j}}\,,
\]
where $j=\pm1$ and $|z|\leq 4^j$, can be expressed in terms of multiple polylogarithms, though they did not present explicitly the formulas of the coefficients in the expressions. However, to the best of our knowledge, in the literature, there are only a few works on the establishment of explicit formulas for general Euler-Ap\'{e}ry-type series directly in terms of (alternating) multiple zeta values. One of the recent such works was due to Chen \cite{ChenKW19}, who used the generalized Arakawa-Kaneko zeta functions and multiple zeta values to establish the explicit formulas of some general Euler-Ap\'{e}ry-type series.

Inspired by the works referred to above, in this paper, by using the methods of iterated integrals and contour integrals, we study some Euler-Ap\'{e}ry-type series of the forms
\begin{equation}\label{EAS.def}
\bms_{i_1i_2\ldots i_k,p}:=\sum_{n=1}^\infty
    \frac{H_n^{(i_1)}H_n^{(i_2)}\cdots H_n^{(i_k)}\binom{2n}{n}}{4^n n^p}\,,\quad
\tilde{\bms}_{i_1i_2\ldots i_k,p}:=\sum_{n=1}^\infty
    \frac{4^nH_n^{(i_1)}H_n^{(i_2)}\cdots H_n^{(i_k)}}{n^p\binom{2n}{n}}\,,
\end{equation}
and of some other similar forms, and establish the corresponding explicit formulas in terms of alternating multiple zeta values. The paper is organized as follows.

In Section \ref{Sec.Int.MZV}, we study the Euler-Ap\'{e}ry-type series by iterated integrals, and establish the elegant explicit formulas for the series
\begin{equation}\label{sums.1234}
\sum_{n=1}^{\infty}\frac{\binom{2n}{n}}{4^nn^p}\,,\quad
\sum_{n=1}^{\infty}\frac{H_n^{(m)}\binom{2n}{n}}{4^nn^p}\,,\quad
\sum_{n=1}^{\infty}\frac{H_nH_n^{(m)}\binom{2n}{n}}{4^nn^p}\,,\quad
\sum_{n=1}^{\infty}\frac{H_n^3\binom{2n}{n}}{4^nn^p}\,,
\end{equation}
and
\begin{equation}\label{sums.56}
\sum_{n=1}^\infty\frac{\ze_n^\star(\{1\}_m)\binom{2n}{n}}{4^nn^p}\,,\quad
\sum_{n=1}^\infty\frac{H_n\ze_n^\star(\{1\}_m)\binom{2n}{n}}{4^nn^p}\,,
\end{equation}
for $m,p\geq1$. As a result, all these series can be expressed directly in terms of alternating multiple zeta values. Some representative examples are as follows:
\begin{align*}
&\sum_{n=1}^\infty\frac{\binom{2n}{n}}{4^n n^{p+1}}
    =-2\ze(\bar{1},\{\hat{1}\}_p)\,,\quad
\sum_{n=1}^\infty\frac{H_n\binom{2n}{n}}{4^n n^{p+1}}
    =-2\ze(\bar{1},\tilde{1},\{\hat{1}\}_p)+2\ze(\tilde{2},\{\hat{1}\}_p)\,,\\
&\sum_{n=1}^\infty\frac{H_n\ze_n^\star(\{1\}_m)\binom{2n}{n}}{4^n n}
    =-2^{m+1}\sum_{k=1}^{m+1}\ze(\bar{k},\widetilde{m+2-k})+2^{m+1}(m+1)\ze(\widetilde{m+2})\,,
\end{align*}
where the related notations and definitions can be found in Section \ref{Sec.MZV}.

In Section \ref{Sec.Con.Int}, we study the Euler-Ap\'{e}ry-type series by contour integral representations and residue computation. More precisely, in Section \ref{Sec.Lem}, we show some lemmas, and in Sections \ref{Sec.Sq.tSq} -- \ref{Sec.tS1q}, by computing the contour integrals related to gamma functions, polygamma functions and trigonometric functions, we establish the expressions of the Euler-Ap\'{e}ry-type series
\begin{equation}\label{sums.789.10}
\sum_{n=1}^\infty\frac{4^n}{n^p\binom{2n}{n}}\,,\quad
\sum_{n=1}^\infty\frac{4^nH_n}{n^p\binom{2n}{n}}\,,\quad
\sum_{n=1}^\infty\frac{4^nH_{2n}}{n^p\binom{2n}{n}}\,,\quad
\sum_{n=1}^\infty\frac{4^nO_n}{n^p\binom{2n}{n}}\,,
\end{equation}
for $p\geq 2$, and
\begin{equation}\label{sums.11.12}
\sum_{n=1}^{\infty}\frac{\binom{2n}{n}}{4^nn^p}\,,\quad
\sum_{n=1}^\infty\frac{H_{2n}\binom{2n}{n}}{4^nn^p}\,,\quad
\sum_{n=1}^\infty\frac{O_n\binom{2n}{n}}{4^nn^p}\,,\quad
\end{equation}
for $p\geq 1$, where $O_n=\sum_{k=1}^n\frac{1}{2k-1}$ are the \emph{odd harmonic numbers}. As a result, we show that the single sums $\bms_{p}=\sum_{n=1}^{\infty}4^{-n}n^{-p}\binom{2n}{n}$ can be expressed in terms of $\ln(2)$ and Riemann zeta values, and all the other sums are reducible to (alternating) multiple zeta values.

Finally, in Section \ref{Sec.Remark}, we introduce briefly the Maple program, which is based on the explicit formulas established in this paper and the \emph{multiple zeta value data mine} due to Bl\"{u}mlein et al. \cite{BBV2010}, and can be used to compute the corresponding Euler-Ap\'{e}ry-type series automatically. Some further remarks are also given there. Moreover, we list all the evaluations of the Euler-Ap\'{e}ry-type series in (\ref{sums.1234}) -- (\ref{sums.11.12}) and of weight $w\leq 6$ in \cite{WangXu19.MP.ESEA}.

Note that the main purpose of this paper is to establish the explicit formulas for some general Euler-Ap\'{e}ry-type series with parameters, so that we can use these formulas to compute the corresponding series immediately. Some special series obtained in this paper by specifying the parameters are not new, and we have tried to list the related sources.

%%%%%%%%%%%%%%%%%%%%%%%%%%%%%%%%%%%%%%%%%%%%%%%%%%%%%%
%%%%%%%%%%%%%%%%%%%%%%%%%%%%%%%%%%%%%%%%%%%%%%%%%%%%%%
%%%%%%%%%%%%%%%%%%%%%%%%%%%%%%%%%%%%%%%%%%%%%%%%%%%%%%
%%%%%%%%%%%%%%%%%%%%%%%%%%%%%%%%%%%%%%%%%%%%%%%%%%%%%%
%%%%%%%%%%%%%%%%%%%%%%%%%%%%%%%%%%%%%%%%%%%%%%%%%%%%%%

\section{Evaluations via iterated integrals and alternating MZVs}\label{Sec.Int.MZV}

In this section, we firstly introduce the (alternating) multiple harmonic sums, the (alternating) multiple zeta values, and related identities, which are important to our study. Next, we establish the explicit formulas of the Euler-Ap\'{e}ry-type series listed in (\ref{sums.1234}) and (\ref{sums.56}) by iterated integrals and alternating multiple zeta values.

%%%%%%%%%%%%%%%%%%%%%%%%%%%%%%%%%%%%%%%%%%%%%%%%%%%%%%
%%%%%%%%%%%%%%%%%%%%%%%%%%%%%%%%%%%%%%%%%%%%%%%%%%%%%%
%%%%%%%%%%%%%%%%%%%%%%%%%%%%%%%%%%%%%%%%%%%%%%%%%%%%%%
%%%%%%%%%%%%%%%%%%%%%%%%%%%%%%%%%%%%%%%%%%%%%%%%%%%%%%
%%%%%%%%%%%%%%%%%%%%%%%%%%%%%%%%%%%%%%%%%%%%%%%%%%%%%%

\subsection{Multiple zeta values and related identities}\label{Sec.MZV}

For positive integers $s_1,s_2,\ldots,s_m$, the \emph{multiple harmonic sums} (MHSs) $\ze_n(s_1,s_2,\ldots,s_m)$ are defined by
\[
\ze_n(s_1,s_2,\ldots,s_m):=
    \sum_{n\geq n_1>n_2>\cdots>n_m\geq 1}\frac{1}{n_1^{s_1}n_2^{s_2}\cdots n_m^{s_m}}\,.
\]
The quantity $m$ is called the \emph{depth} of a multiple harmonic sum, and the quantity $w:=s_1+\cdots+s_m$ is called the \emph{weight}. By convention, $\ze_n(s_1,s_2,\ldots,s_m)=0$ for $n<m$, $\ze_n(\emptyset)=1$, and we denote $m$ repetitions of a substring by $\{\cdots\}_m$, with $\{\cdots\}_0:=\emptyset$. When taking the limit $n\rightarrow\infty$, we get the so-called \emph{multiple zeta values} (MZVs):
\begin{equation}\label{def.MZV}
\ze(s_1,s_2,\ldots,s_m)=\lim_{n\to\infty}\ze_n(s_1,s_2,\ldots,s_m)
    =\sum_{n_1>n_2>\cdots>n_m\geq 1}\frac{1}{n_1^{s_1}n_2^{s_2}\cdots n_m^{s_m}}\,,
\end{equation}
defined for $s_1>1$ to ensure the convergence of the series. Additionally, similarly to MHSs, denote the \emph{multiple harmonic star sums} $\ze_n^\star(s_1,s_2,\ldots,s_m)$ by
\[
\ze_n^\star(s_1,s_2,\ldots,s_m):=
    \sum_{n\geq n_1\geq n_2\geq\cdots\geq n_m\geq 1}
    \frac{1}{n_1^{s_1}n_2^{s_2}\cdots n_m^{s_m}}\,.
\]

The \emph{alternating multiple harmonic sums} are defined by
\[
\ze_n(s_1,s_2,\ldots,s_m;\si_1,\si_2,\ldots,\si_m)=\sum_{n\geq n_1>n_2>\cdots>n_m\geq 1}
    \frac{\si_1^{n_1}\si_2^{n_2}\cdots \si_m^{n_m}}{n_1^{s_1}n_2^{s_2}\cdots n_m^{s_m}}\,,
\]
where $s_j$ are positive integers, $\si_j=\pm1$, for $j=1,2,\ldots,m$, with $(s_1,\si_1)\neq(1,1)$. The limit cases of alternating MHSs give rise to \emph{alternating multiple zeta values}
\begin{equation}\label{def.AMZV}
\ze(s_1,s_2,\ldots,s_m;\si_1,\si_2,\ldots,\si_m)
    =\lim_{n\to\infty}\ze_n(s_1,s_2,\ldots,s_m;\si_1,\si_2,\ldots,\si_m)\,.
\end{equation}
As usual, when writing alternating MHSs and MZVs, we shall combine the strings of exponents and signs into a single string, with $s_j$ in the $j$th position when $\si_j=+1$, and $\bar{s}_j$ in the $j$th position when $\si_j=-1$. For example,
\[
\ze(\bar{s}_1,s_2,\ldots,\bar{s}_m)=\sum_{n_1>n_2>\cdots>n_m\geq 1}
    \frac{(-1)^{n_1+n_m}}{n_1^{s_1}n_2^{s_2}\cdots n_m^{s_m}}\,.
\]
In particular, it is known that $\ze_n(\bar{s}):=-\bar{H}_n^{(s)}$ are the alternating harmonic numbers, and
\[
\ze(\bar{s})=\sum_{n=1}^{\infty}\frac{(-1)^n}{n^s}=(2^{1-s}-1)\ze(s)\,,
\]
with $\ze(\bar{1})=-\ln(2)$. Moreover, for brevity, we also use the symbols
\[
\ze_n(\sigma_1s_1,\sigma_2s_2,\ldots,\sigma_ms_m)\quad\text{and}\quad \ze(\sigma_1s_1,\sigma_2s_2,\ldots,\sigma_ms_m)
\]
in the sequel to denote the alternating MHSs and MZVs. Henceforth, if $\sigma_j=-1$, then we use ${\bar s_j}$ to replace $-s_j$. As an instance, $\ze(-2,3,-1,4)=\ze(\bar{2},3,\bar{1},4)$.

The systematic study of MZVs and alternating MZVs began in the early 1990s with the works of Hoffman \cite{Hoff92} and Zagier \cite{Zag92}, and rational relations among alternating MZVs are tabulated in the multiple zeta value data mine up to the weight 12 by Bl\"{u}mlein et al. \cite{BBV2010}.

Now, we present some identities on (alternating) MHSs, which will be frequently used in this paper. According to \cite[Section 1]{KuPro10} and \cite[Theorem 2.5]{Xu17.MZVES}, the following generating function related to the unsigned Stirling numbers of the first kind $\tbnq{n}{k}$ and the MHSs holds:
\begin{equation}\label{gf.Stir}
\frac{1}{k!}(-\ln(1-t))^k
    =\sum_{n=1}^{\infty}\binq{n}{k}\frac{t^n}{n!}
    =\sum_{n=1}^{\infty}\ze_{n-1}(\{1\}_{k-1})\frac{t^n}{n}\,,\quad\text{for }k\geq 1\,.
\end{equation}
By the definition of MHSs, we have
\begin{align}
\ze_{n-1}(s_1,s_2,\ldots,s_m)
    &=\ze_n(s_1,s_2,\ldots,s_m)-\frac{1}{n^{s_1}}\ze_{n-1}(s_2,s_3,\ldots,s_m)\nonumber\\
    &=\cdots=\sum_{l=1}^{m+1}\frac{(-1)^{l-1}}{n^{s_1+s_2+\cdots+s_{l-1}}}
        \ze_n(s_l,s_{l+1},\ldots,s_m)\,,\label{rec.MHS}
\end{align}
for $n,m\geq 1$, where, by convention, $s_1+\cdots+s_{l-1}:=0$ if $l=1$. Moreover, we find that
\begin{align}
&\ze_n(s_1,s_2,\ldots,s_m)\nonumber\\
&\quad=2^{s_1+s_2+\cdots+s_m-m}\sum_{2n\geq n_1>n_2>\cdots>n_m\geq1}
    \frac{(1+(-1)^{n_1})(1+(-1)^{n_2})\cdots(1+(-1)^{n_m})}
        {n_1^{s_1}n_2^{s_2}\cdots n_m^{s_m}}\nonumber\\
&\quad=2^{s_1+s_2+\cdots+s_m-m}\sum_{\substack{\si_j\in\{\pm 1\}\\j=1,2,\ldots,m}}
    \ze_{2n}(s_1\si_1,s_2\si_2,\ldots,s_m\si_m)\,.\label{zen.ze2n}
\end{align}
For convenience, in (\ref{def.MZV}), if there is a ``$1+(-1)^{n_j}$'' (or ``$1-(-1)^{n_j}$'') in the numerator of the summand, we put a ``hat'' (or ``tilde'') on the top of $s_j$. For example,
\[
\ze(\hat{s}_1,s_2,\tilde{s}_3,\bar{s}_4)
    =\sum_{n_1>n_2>n_3>n_4\geq 1}
    \frac{1+(-1)^{n_1}}{n_1^{s_1}}\cdot\frac{1}{n^2}\cdot\frac{1-(-1)^{n_3}}{n_3^{s_3}}
    \cdot\frac{(-1)^{n_4}}{n_4^{s_4}}\,,
\]
and we call it a \emph{mixed multiple zeta value}. It is clear that all mixed MZVs can be expressed as rational linear combinations of (alternating) MZVs. In particular, we have
\[
\ze(\tilde{s}_1,\tilde{s}_2,\ldots,\tilde{s}_m)=2^m t(s_1,s_2,\ldots,s_m)\,,
\]
where $t(s_1,s_2,\ldots,s_m)$ are the \emph{multiple $t$-values} or the \emph{Hoffman $t$-values} \cite{Hoff16}, defined by
\begin{align*}
t(s_1,s_2,\ldots,s_m):
&=\sum_{\substack{n_1>n_2>\cdots>n_m\geq 1\\n_i\ {\rm odd}}}
    \frac{1}{n_1^{s_1}n_2^{s_2}\cdots n_m^{s_m}}\nonumber\\
&=\sum_{n_1>n_2>\cdots>n_m\geq 1}\frac{1}{(2n_1-1)^{s_1}(2n_2-1)^{s_2}\cdots(2n_m-1)^{s_m}}\,.
\end{align*}

%%%%%%%%%%%%%%%%%%%%%%%%%%%%%%%%%%%%%%%%%%%%%%%%%%%%%%
%%%%%%%%%%%%%%%%%%%%%%%%%%%%%%%%%%%%%%%%%%%%%%%%%%%%%%
%%%%%%%%%%%%%%%%%%%%%%%%%%%%%%%%%%%%%%%%%%%%%%%%%%%%%%
%%%%%%%%%%%%%%%%%%%%%%%%%%%%%%%%%%%%%%%%%%%%%%%%%%%%%%
%%%%%%%%%%%%%%%%%%%%%%%%%%%%%%%%%%%%%%%%%%%%%%%%%%%%%%

\subsection{Single sums $\bms_p$}\label{Sec.Sp}

Firstly, we give a lemma, by which an integral can be transformed into a summation:

\begin{lemma}\label{Lem.int.Inm}
For integers $n\geq 1$ and $m\geq 0$, the following integral satisfies
\[
I(n,m):=\int_0^1 t^{n-1}\ln^m(t)\ln\left(\frac{2}{1+t}\right)\ud t
    =(-1)^mm!\sum_{j=1}^{m+1}\frac{(-1)^n}{n^{m+2-j}}\{\ze_n(\bar j)-\ze(\bar j)\}\,.
\]
\end{lemma}

\begin{proof}
Expanding $\ln(1+t)$, using the integral
\begin{equation}\label{int.tn.lnm.1}
\int_0^1t^{n-1}\ln^m(t)\uud t=\frac{(-1)^mm!}{n^{m+1}}\,,
\end{equation}
and then performing partial fraction decomposition, we have
\begin{align*}
\int_0^1t^{n-1}\ln^m(t)\ln(1+t)\uud t
    &=(-1)^mm!\sum_{k=1}^\infty\frac{(-1)^{k-1}}{k(k+n)^{m+1}}\\
    &=(-1)^mm!\left\{\frac{1}{n^{m+1}}\ln(2)+\sum_{j=1}^{m+1}
        \frac{1}{n^{m+2-j}}\sum_{k=1}^{\infty}\frac{(-1)^k}{(k+n)^j}\right\}\,,
\end{align*}
which, together with the definitions of alternating harmonic numbers and zeta values, gives the desired result.
\end{proof}

In particular, we have
\[
I(n,0)=\frac{(-1)^{n+1}}{n}(\ze(\bar{1})-\ze_n(\bar{1}))\,.
\]
By this lemma, the explicit formula of the single sums can be established.

\begin{theorem}\label{Th.Sp}
For integer $p\geq 0$, the single sums $\bms_{p+1}$ satisfy
\begin{align}\label{Sp}
\bms_{p+1}:=\sum_{n=1}^\infty\frac{\binom{2n}{n}}{4^n n^{p+1}}
    =-2\ze(\bar{1},\{\hat{1}\}_p)\,.
\end{align}
\end{theorem}

\begin{proof}
From \cite{ChenH16,Leh1985}, it is known that
\begin{equation}\label{gf.cb1}
\sum_{n=1}^\infty\frac{\binom{2n}{n}}{4^n n}t^n
    =2\ln\left(\frac{2}{1+\sqrt{1-t}}\right)\,.
\end{equation}
Multiplying it by $\frac{\ln^{p-1}(t)}{t}$, for $p\geq 1$, integrating over $(0,1)$, and using Eq. (\ref{int.tn.lnm.1}) yield
\[
\sum_{n=1}^\infty\frac{\binom{2n}{n}}{4^n n^{p+1}}
=\frac{(-1)^{p-1}}{(p-1)!}\cdot2\int_0^1\frac{\ln^{p-1}(1-t)}{1-t}
    \ln\left(\frac{2}{1+\sqrt{t}}\right)\uud t\,.
\]
When $p\geq 2$, by (\ref{gf.Stir}), we have
\begin{align*}
\sum_{n=1}^\infty\frac{\binom{2n}{n}}{4^n n^{p+1}}
&=2\sum_{n=1}^\infty\frac{\ze_{n-1}(\{1\}_{p-2})}{n}
    \int_0^1\frac{t^n}{1-t}\ln\smbb{\frac{2}{1+\sqrt{t}}}\uud t\\
&=4\sum_{n=1}^\infty\sum_{k=0}^\infty\frac{\ze_{n-1}(\{1\}_{p-2})}{n}
    \int_0^1t^{2n+2k+1}\ln\smbb{\frac{2}{1+t}}\uud t\\
&=4\sum_{n=1}^\infty\ze_{n-1}(\{1\}_{p-1})\int_0^1 t^{2n-1}\ln\smbb{\frac{2}{1+t}}\uud t\,,
\end{align*}
where the last term also holds for $p=1$. Next, using Eqs. (\ref{rec.MHS}) and (\ref{zen.ze2n}), and applying Lemma \ref{Lem.int.Inm}, we further transform the last term to
\begin{align}
\sum_{n=1}^\infty\frac{\binom{2n}{n}}{4^n n^{p+1}}
&=4\sum_{l=1}^p(-1)^{l-1}\sum_{\substack{\sigma_j\in\{\pm 1\}\\j=1,2,\ldots,p-l}}
    \sum_{n=1}^\infty\frac{1}{n^{l-1}}\ze_{2n}(\si_1,\si_2,\ldots,\si_{p-l})I(2n,0)\nonumber\\
&=\sum_{l=1}^p(-1)^{l-1}2^l\sum_{\substack{\si_j\in\{\pm 1\}\\j=1,2,\ldots,p-l}}
    \sum_{n=1}^\infty\frac{1+(-1)^n}{n^{l-1}}\ze_n(\si_1,\si_2,\ldots,\si_{p-l})I(n,0)\nonumber\\
&=\sum_{l=1}^p(-2)^l\sum_{\substack{\si_j\in\{\pm 1\}\\j=0,1,\ldots,p-l}}
    \sum_{n=1}^\infty\frac{\si_0^n}{n^l}\ze_n(\si_1,\si_2,\ldots,\si_{p-l})
    (\ze(\bar{1})-\ze_n(\bar{1}))\nonumber\\
&=\sum_{l=1}^{p-1}(-2)^l\sum_{\substack{\si_j\in\{\pm 1\}\\j=0,1,\ldots,p-l}}
    \bibb{
        \begin{array}{c}
            \ze(\bar{1},l\si_0,\si_1,\si_2,\ldots,\si_{p-l})\\
            +\ze(\bar{1},(l+1)\si_0\si_1,\si_2,\ldots,\si_{p-l})
        \end{array}
    }\nonumber\\
&\quad+(-2)^p(\ze(\bar{1},p)+\ze(\bar{1},\bar{p}))\,.\label{Sp.temp1}
\end{align}
Since
\begin{align*}
&(-2)^l\sum_{\substack{\si_j\in\{\pm 1\}\\j=0,1,\ldots,p-l}}
    \ze(\bar{1},(l+1)\si_0\si_1,\si_2,\ldots,\si_{p-l})\\
&\quad=(-2)^l\cdot2\sum_{\substack{\tau_j\in\{\pm 1\}\\j=0,1,\ldots,p-l-1}}
    \ze(\bar{1},(l+1)\tau_0,\tau_1,\ldots,\tau_{p-l-1})\,,
\end{align*}
by telescoping, (\ref{Sp.temp1}) can be simplified as
\[
\sum_{n=1}^\infty\frac{\binom{2n}{n}}{4^n n^{p+1}}
=-2\sum_{\substack{\si_j\in\{\pm 1\}\\j=0,1,\ldots,p-1}}
    \ze(\bar{1},\si_0,\si_1,\ldots,\si_{p-1})
=-2\ze(\bar{1},\{\hat{1}\}_p)\,.
\]
Thus, Eq. (\ref{Sp}) holds for $p=1$. Setting $t=1$ in (\ref{gf.cb1}) gives $\bms_1=2\ln(2)=-2\ze(\bar{1})$, which means that (\ref{Sp}) also holds for $p=0$. Hence, we obtain this theorem.
\end{proof}

\begin{example}
Setting $p=1$ in Theorem \ref{Th.Sp} yields
\[
\sum_{n=1}^\infty\frac{\binom{2n}{n}}{4^n n^2}
    =-2\{\ze(\bar{1},1)+\ze(\bar{1},\bar{1})\}=\ze(2)-2\ln^2(2)\,,
\]
which was given by Boyadzhiev \cite[Section 3]{Boyad12} and Sun \cite[Remark 5.2]{Sun15.ANSR}. Similarly, we have
\begin{align*}
\sum_{n=1}^\infty\frac{\binom{2n}{n}}{4^n n^3}
    &=2\ze(3)-2\ze(2)\ln(2)+\frac{4}{3}\ln^3(2)\,,\\
\sum_{n=1}^\infty\frac{\binom{2n}{n}}{4^n n^4}
    &=\frac{9}{4}\ze(4)-4\ze(3)\ln(2)+2\ze(2)\ln^2(2)-\frac{2}{3}\ln^4(2)\,.
\end{align*}
\end{example}

%%%%%%%%%%%%%%%%%%%%%%%%%%%%%%%%%%%%%%%%%%%%%%%%%%%%%%
%%%%%%%%%%%%%%%%%%%%%%%%%%%%%%%%%%%%%%%%%%%%%%%%%%%%%%
%%%%%%%%%%%%%%%%%%%%%%%%%%%%%%%%%%%%%%%%%%%%%%%%%%%%%%
%%%%%%%%%%%%%%%%%%%%%%%%%%%%%%%%%%%%%%%%%%%%%%%%%%%%%%
%%%%%%%%%%%%%%%%%%%%%%%%%%%%%%%%%%%%%%%%%%%%%%%%%%%%%%

\subsection{Linear sums $\bms_{m,p}$}

For convenience, denote
\[
\int_0^1 f_1(t)\uud tf_2(t)\uud t\cdots f_k(t)\uud t
    :=\int\limits_{0<t_k<\cdots<t_1<1}f_1(t_1)f_2(t_2)\cdots f_k(t_k)
    \uud t_1\uud t_2\cdots\uud t_k\,.
\]
Now, let us establish the explicit formula of the linear sums $\bms_{m,p}$.

\begin{theorem}\label{Th.Smp}
For integers $m\geq 1$ and $p\geq 0$, the linear sums $\bms_{m+1,p+1}$ satisfy
\begin{equation}\label{Smp}
\bms_{m+1,p+1}:=\sum_{n=1}^\infty\frac{H_n^{(m+1)}\binom{2n}{n}}{4^n n^{p+1}}
    =4\ze(\bar{1},\{\hat{1}\}_p,\hat{2},\{\hat{1}\}_{m-1})
        -2\ze(m+1)\ze(\bar{1},\{\hat{1}\}_p)\,.
\end{equation}
\end{theorem}

\begin{proof}
When $p\geq 1$, using the generating function (\ref{gf.cb1}) and the integral
\begin{equation}\label{int.tn.lnm.2}
\int_0^1\frac{t^n\ln^m(t)}{1-t}\uud t
    =\sum_{i=0}^{\infty}\int_0^1t^{n+i}\ln^m(t)\uud t
    =(-1)^mm!\{\ze(m+1)-H_n^{(m+1)}\}\,,
\end{equation}
and then applying the change of variables $t_j\mapsto 1-t_{p+2-j}$, for $j=1,2,\ldots,p+1$, we have
\begin{align*}
\sum_{n=1}^\infty\frac{(\ze(m+1)-H_n^{(m+1)})\binom{2n}{n}}{4^n n^{p+1}}
&=\frac{(-1)^m}{m!}\cdot2
    \int_0^1\frac{\ln^{m}(t)}{1-t}\uud t
    \underbrace{\frac{\ud t}{t}\cdots\frac{\ud t}{t}}_{p-1}
    \frac{\ln\smbb{\frac{2}{1+\sqrt{1-t}}}}{t}\uud t\\
&=\frac{(-1)^m}{m!}\cdot2
    \int_0^1\frac{\ln\smbb{\frac{2}{1+\sqrt{t}}}}{1-t}\uud t
    \underbrace{\frac{\ud t}{1-t}\cdots\frac{\ud t}{1-t}}_{p-1}
    \frac{\ln^{m}(1-t)}{t}\uud t\,.
\end{align*}
If $p\geq 2$, substituting Eq. (\ref{gf.Stir}) and the power series expansion of $1/(1-t)$ gives
\begin{align}
&\sum_{n=1}^\infty\frac{(\ze(m+1)-H_n^{(m+1)})\binom{2n}{n}}{4^n n^{p+1}}\nonumber\\
&\quad=2\sum_{n_1,n_2,\ldots,n_p\geq 1}
    \frac{\ze_{n_1-1}(\{1\}_{m-1})}{n_1^2(n_1+n_2)\cdots(n_1+n_2+\cdots+n_p)}
    \int_0^1t^{n_1+n_2+\cdots+n_p}\frac{\ln\smbb{\frac{2}{1+\sqrt{t}}}}{1-t}\uud t\nonumber\\
&\quad=2\sum_{n=1}^\infty\frac{1}{n}\ze_{n-1}(\{1\}_{p-2},2,\{1\}_{m-1})
    \int_0^1t^n\frac{\ln\smbb{\frac{2}{1+\sqrt{t}}}}{1-t}\uud t\nonumber\\
&\quad=4\sum_{n=1}^\infty\ze_{n-1}(\{1\}_{p-1},2,\{1\}_{m-1})
    \int_0^1t^{2n-1}\ln\smbb{\frac{2}{1+t}}\uud t\,.\label{Smp.temp1}
\end{align}
By computation, the last term of (\ref{Smp.temp1}) also holds for $p=1$. Now, according to (\ref{rec.MHS}), we have
\begin{align*}
&\ze_{n-1}(\{1\}_{p-1},2,\{1\}_{m-1})\\
&\quad=\sum_{l=1}^p\frac{(-1)^{l-1}}{n^{l-1}}\ze_n(\{1\}_{p-l},2,\{1\}_{m-1})
    +\sum_{l=1}^m\frac{(-1)^{l+p-1}}{n^{l+p}}\ze_n(\{1\}_{m-l})\,,
\end{align*}
which can be used to split the summation in (\ref{Smp.temp1}) into two parts. By applying Eq.  (\ref{zen.ze2n}) and Lemma \ref{Lem.int.Inm}, and then performing telescoping, we can rewrite the first part as
\begin{align*}
&4\sum_{l=1}^p(-1)^{l-1}\sum_{n=1}^{\infty}\frac{1}{n^{l-1}}
    \ze_n(\{1\}_{p-l},2,\{1\}_{m-1})I(2n,0)\\
&\quad=8\sum_{l=1}^p(-1)^{l-1}\sum_{n=1}^{\infty}\frac{1}{n^{l-1}}
    \sum_{\si_j\in\{\pm1\}}
    \ze_{2n}\smbb{\cat_{i=1}^{p-l}\{\si_i\},2\si_{p-l+1},
        \cat_{i=p-l+2}^{p-l+m}\{\si_i\}}I(2n,0)\\
&\quad=2\sum_{l=1}^p(-2)^l\sum_{\si_j\in\{\pm1\}}
    \sum_{n=1}^{\infty}\frac{\si_0^n}{n^l}
    \ze_n\smbb{\cat_{i=1}^{p-l}\{\si_i\},2\si_{p-l+1},
        \cat_{i=p-l+2}^{p-l+m}\{\si_i\}}(\ze(\bar{1})-\ze_n(\bar{1}))\\
&\quad=2\sum_{l=1}^{p-1}(-2)^l\sum_{\si_j\in\{\pm1\}}
    \bibb{
        \begin{array}{c}
            \ze\smbb{\bar{1},l\si_0,\cat\limits_{i=1}^{p-l}\{\si_i\},2\si_{p-l+1},
                \cat\limits_{i=p-l+2}^{p-l+m}\{\si_i\}}\\
            +\ze\smbb{\bar{1},(l+1)\si_0\si_1,\cat\limits_{i=2}^{p-l}\{\si_i\},2\si_{p-l+1}
                \cat\limits_{i=p-l+2}^{p-l+m}\{\si_i\}}
        \end{array}
    }\\
&\quad\quad+2(-2)^p\sum_{\si_j\in\{\pm1\}}
    \{\ze(\bar{1},p\si_0,2\si_1,\si_2,\ldots,\si_m)
        +\ze(\bar{1},(p+2)\si_0\si_1,\si_2,\ldots,\si_m)\}\\
&\quad=-4\sum_{\si_j\in\{\pm 1\}}
    \ze(\bar{1},\si_0,\si_1,\ldots,\si_{p-1},2\si_p,\si_{p+1},\ldots,\si_{p+m-1})\\
&\quad\quad+4(-2)^p\sum_{\si_j\in\{\pm1\}}
    \ze(\bar{1},(p+2)\si_0,\si_1,\ldots,\si_{m-1})\,,
\end{align*}
where $\cat_{i=j}^k\{a_i\}$ abbreviates the concatenated argument sequence $a_{j},a_{j+1},\ldots,a_{k}$, with the convention $\cat_{i=j}^k\{a_i\}:=\emptyset$ for $k<j$. The second part can also be evaluated:
\begin{align*}
&4\sum_{l=1}^m(-1)^{l+p-1}\sum_{n=1}^{\infty}\frac{1}{n^{l+p}}\ze_n(\{1\}_{m-l})I(2n,0)\\
&\quad=2\sum_{l=1}^m(-2)^{l+p}\sum_{\si_j\in\{\pm1\}}
    \sum_{n=1}^{\infty}\frac{\si_0^n}{n^{l+p+1}}
    \ze_n(\si_1,\si_2,\ldots,\si_{m-l})(\ze(\bar{1})-\ze_n(\bar{1}))\\
&\quad=2\sum_{l=1}^{m-1}(-2)^{l+p}\sum_{\si_j\in\{\pm1\}}
    \bibb{
        \begin{array}{c}
            \ze(\bar{1},(l+p+1)\si_0,\si_1,\si_2,\ldots,\si_{m-l})\\
            +\ze(\bar{1},(l+p+2)\si_0\si_1,\si_2,\ldots,\si_{m-l})
        \end{array}
    }\\
&\quad\quad+2(-2)^{m+p}\{\ze(\bar{1},m+p+1)+\ze(\bar{1},\ol{m+p+1})\}\\
&\quad=2(-2)^{p+1}\sum_{\si_j\in\{\pm1\}}\ze(\bar{1},(p+2)\si_0,\si_1,\si_2,\ldots,\si_{m-1})\,.
\end{align*}
Combining the last two equations yields, for $p\geq 1$, the explicit formula
\begin{align}
\sum_{n=1}^\infty\frac{(\ze(m+1)-H_n^{(m+1)})\binom{2n}{n}}{4^n n^{p+1}}
&=-4\sum_{\substack{\si_j\in\{\pm 1\}\\j=0,1,\ldots,p+m-1}}
    \ze\smbb{\bar{1},\cat_{i=0}^{p-1}\{\si_i\},2\si_p,\cat_{i=p+1}^{p+m-1}\{\si_i\}}\nonumber\\
&=-4\ze(\bar{1},\{\hat{1}\}_p,\hat{2},\{\hat{1}\}_{m-1})\,.\label{Smp.temp2}
\end{align}
When $p=0$, following steps analogous to those above, we have
\begin{align*}
&\sum_{n=1}^\infty\frac{(\ze(m+1)-H_n^{(m+1)})\binom{2n}{n}}{4^n n}
    =\frac{(-1)^m}{m!}\cdot
        2\int_0^1\ln\smbb{\frac{2}{1+\sqrt{1-t}}}\frac{\ln^m(t)}{1-t}\uud t\\
&\quad=2\sum_{l=1}^m(-2)^l\sum_{\si_j\in\{\pm1\}}
    \sum_{n=1}^{\infty}\frac{\si_0^n}{n^{l+1}}
    \ze_n(\si_1,\ldots,\si_{m-l})(\ze(\bar{1})-\ze_n(\bar{1}))\\
&\quad=-4\sum_{\si_j\in\{\pm 1\}}\ze(\bar{1},2\si_0,\si_1,\si_2,\ldots,\si_{m-1})
    =-4\ze(\bar{1},\hat{2},\{\hat{1}\}_{m-1})\,,
\end{align*}
which coincides with the form of (\ref{Smp.temp2}). Thus, by Theorem \ref{Th.Sp}, we obtain the final result.
\end{proof}

\begin{example}
When $m=1$, setting $p=0,1$ in Theorem \ref{Th.Smp} yields
\begin{align*}
\sum_{n=1}^\infty\frac{H_n^{(2)}\binom{2n}{n}}{4^n n}
    &=4\ze(\bar{1},2)+4\ze(\bar{1},\bar{2})-2\ze(2)\ze(\bar{1})=\frac{3}{2}\ze(3)\,,\\
\sum_{n=1}^\infty\frac{H_n^{(2)}\binom{2n}{n}}{4^n n^2}
    &=3\ze(4)-3\ze(3)\ln(2)\,,
\end{align*}
where the value of $\bms_{2,1}$ was shown by Sun \cite[Remark 5.2]{Sun15.ANSR} by using Mathematica.\hfill\qedsymbol
\end{example}

%%%%%%%%%%%%%%%%%%%%%%%%%%%%%%%%%%%%%%%%%%%%%%%%%%%%%%
%%%%%%%%%%%%%%%%%%%%%%%%%%%%%%%%%%%%%%%%%%%%%%%%%%%%%%
%%%%%%%%%%%%%%%%%%%%%%%%%%%%%%%%%%%%%%%%%%%%%%%%%%%%%%
%%%%%%%%%%%%%%%%%%%%%%%%%%%%%%%%%%%%%%%%%%%%%%%%%%%%%%
%%%%%%%%%%%%%%%%%%%%%%%%%%%%%%%%%%%%%%%%%%%%%%%%%%%%%%

\subsection{Linear sums $\bms_{1,p}$ and $\bms_{m,p}^\star$}

The \emph{exponential complete Bell polynomials} $Y_n$ are defined by
\begin{equation}\label{cBell.gf}
\exp\smbb{\sum_{k=1}^{\infty}x_k\frac{t^k}{k!}}
    =\sum_{n=0}^{\infty}Y_n(x_1,x_2,\ldots,x_n)\frac{t^n}{n!}\,,
\end{equation}
and satisfy the recurrence
\begin{equation}\label{cBell.rec}
Y_n(x_1,x_2,\ldots,x_n)=\sum_{j=0}^{n-1}\binom{n-1}{j}x_{n-j}Y_j(x_1,x_2,\ldots,x_j)
    \,,\quad n\geq1\,;
\end{equation}
see \cite[Section 3.3]{Com74} and \cite[Section 2.8]{Riordan58}. From \cite[Eqs. (2.5) and (2.9)]{Xu17.MZVES}, it is known that
\begin{equation}\label{zn*1k}
\int_0^1t^{n-1}\ln^k(1-t)\uud t
    =(-1)^kk!\frac{\ze_n^\star(\{1\}_k)}{n}\,,
\end{equation}
for $n\geq1$ and $k\geq0$, and $\ze_n^\star(\{1\}_k)=\frac{1}{k!}Y_k(H_n,1!H_n^{(2)},2!H_n^{(3)},\ldots,(k-1)!H_n^{(k)})$\,. Thus we can use (\ref{cBell.rec}) to evaluate $\ze_n^\star(\{1\}_k)$ and the integral in (\ref{zn*1k}) rapidly. For example, we have
\begin{align*}
&\ze_n^\star(1)=H_n\,,\quad
    \ze_n^\star(1,1)=\frac{1}{2}(H_n^2+H_n^{(2)})\,,\quad
    \ze_n^\star(1,1,1)=\frac{1}{6}(H_n^3+3H_nH_n^{(2)}+2H_n^{(3)})\,,\\
&\ze_n^\star(1,1,1,1)=\frac{1}{24}(H_n^4+6H_n^2H_n^{(2)}+3(H_n^{(2)})^2+8H_nH_n^{(3)}+6H_n^{(4)})\,.
\end{align*}
The following theorem presents the explicit formulas of a kind of series involving
$\ze_n^\star (\{1\}_m)$.

\begin{theorem}\label{Th.S*mp}
For integers $m,p\geq 0$, the linear sums $\bms_{m,p}^\star$ satisfy
\begin{align}
&\bms_{m,1}^\star:=\sum_{n=1}^\infty\frac{\ze_n^\star(\{1\}_m)\binom{2n}{n}}{4^n n}
    =-2^{m+1}\ze(\ol{m+1})\,,\label{S*m1}\\
&\bms_{m,p+2}^\star:=\sum_{n=1}^\infty\frac{\ze_n^\star(\{1\}_m)\binom{2n}{n}}{4^n n^{p+2}}
    =-2^{m+1}\sum_{\substack{|k|_{p+1}\leq w-1\\k_1,\ldots,k_{p+1}\geq 1}}
        \ze\smbb{\bar{k}_{p+1},\widehat{w-|k|_{p+1}},\cat_{i=1}^p\{\hat{k}_i\}}
        \,,\label{S*mp}
\end{align}
where $|k|_l:=k_1+k_2+\cdots+k_l$, and $w:=m+p+2$ is the weight of the corresponding sums.
\end{theorem}

\begin{proof}
By considering the integral (\ref{zn*1k}) and the generating function
\[
\sum_{n=0}^\infty\frac{\binom{2n}{n}}{4^n}t^n=\frac{1}{\sqrt{1-t}}\,,
\]
we establish Eq. (\ref{S*m1}):
\begin{align*}
\sum_{n=1}^\infty\frac{\ze_n^\star(\{1\}_m)\binom{2n}{n}}{4^n n}
    &=(-1)^mm!\int_0^1\smbb{\frac{1}{\sqrt{1-t}}-1}\frac{\ln^m(1-t)}{t}\uud t\\
    &=(-1)^mm!\cdot2^{m+1}\int_0^1\frac{\ln^m(t)}{1+t}\uud t
        =-2^{m+1}\ze(\ol{m+1})\,.
\end{align*}
The proof of Eq. (\ref{S*mp}) is similar to that of (\ref{Smp}), but more complex, so we present the crucial steps. When $p\geq 1$, using (\ref{gf.cb1}) and (\ref{zn*1k}), we obtain the integral representation of the sum, which can be further transformed respectively by the integral
\begin{align*}
\int_0^xt^{n-1}\ln^m(t)\uud t
    =\sum_{l=0}^m\frac{(-1)^ll!}{n^{l+1}}\binom{m}{l}\ln^{m-l}(x)x^n\,,
    \quad\text{for } x\in(0,1)\,,
\end{align*}
Eqs. (\ref{rec.MHS}) and (\ref{zen.ze2n}), and Lemma \ref{Lem.int.Inm} to the following forms:
\begin{align*}
&\sum_{n=1}^\infty\frac{\ze_n^\star(\{1\}_m)\binom{2n}{n}}{4^n n^{p+2}}
=\frac{(-1)^m}{m!}\cdot 2\int_0^1\frac{\ln^m(1-t)}{t}\uud t
    \underbrace{\frac{\ud t}{t}\cdots\frac{\ud t}{t}}_{p-1}
    \frac{\ln\smbb{\frac{2}{1+\sqrt{1-t}}}}{t}\uud t\\
&\quad=(-1)^m\sum_{0\leq |k|_p\leq m}\frac{(-1)^{|k|_p}}{(m-|k|_p)!}2^{m+2-|k|_p}
    \sum_{n=1}^{\infty}\ze_{n-1}(k_1+1,\ldots,k_p+1)I(2n,m-|k|_p)\\
&\quad=(-1)^m\sum_{0\leq |k|_p\leq m}\frac{(-1)^{|k|_p}}{(m-|k|_p)!}
    \sum_{l=1}^{p+1}(-1)^{l-1}2^{m+l}
    \sum_{n=1}^{\infty}\frac{1+(-1)^n}{n^{|k|_{l-1}+l-1}}\\
&\quad\quad\times\sum_{\si_i\in\{\pm1\}}
    \ze_n((k_l+1)\si_l,\ldots,(k_p+1)\si_p)I(n,m-|k|_p)\\
&\quad=\sum_{0\leq |k|_{p+1}\leq m}\sum_{l=1}^{p+1}(-1)^l2^{m+l}
    \sum_{\si_i\in\{\pm1\}}\sum_{n=1}^{\infty}
    \frac{\si_0^n\ze_n(\cat_{i=l}^p\{(k_i+1)\si_i\})}{n^{m+l-k_l-\cdots-k_{p+1}}}
    \sum_{j=n+1}^{\infty}\frac{(-1)^j}{j^{k_{p+1}+1}}\\
&\quad=-2^{m+1}\sum_{0\leq|k|_{p+1}\leq m}\sum_{\si_i\in\{\pm1\}}
    \ze\smbb{\ol{k_{p+1}+1},(m+1-|k|_{p+1})\si_0,\cat_{i=1}^p\{(k_i+1)\si_i\}}\,,
\end{align*}
where $k_i\geq 0$, for $i=1,2,\ldots,p+1$. Thus, by the change of variables and the definition of mixed MZVs, we obtain the formula (\ref{S*mp}). Finally, when $p=0$, we have
\begin{align*}
\sum_{n=1}^\infty\frac{\ze_n^\star(\{1\}_m)\binom{2n}{n}}{4^n n^2}
&=\frac{(-1)^m}{m!}\cdot2\int_0^1\frac{\ln^m(1-t)}{t}\ln\smbb{\frac{2}{1+\sqrt{1-t}}}\uud t\\
&=\frac{(-1)^m}{m!}\cdot2^{m+2}\sum_{k=1}^{\infty}I(2k,m)
    =-2^{m+1}\sum_{j=1}^{m+1}\sum_{k>n\geq1}\frac{1+(-1)^n}{n^{m+2-j}}
        \cdot\frac{(-1)^k}{k^j}\,,
\end{align*}
which indicates that the formula (\ref{S*mp}) also holds for $p=0$. This completes the proof.
\end{proof}

Setting $m=0$ in Theorem \ref{Th.S*mp} gives Theorem \ref{Th.Sp}. Setting $m=1$ in Theorem \ref{Th.S*mp} and considering all the positive integer solutions of the inequality $k_1+\cdots+k_{p+1}\leq p+2$, we establish the explicit formula of the linear sums $\bms_{1,p+1}$, which is the supplement to Theorem \ref{Th.Smp}.

\begin{corollary}\label{Coro.S1p}
For integer $p\geq 0$, the linear sums $\bms_{1,p+1}$ satisfy
\[
\bms_{1,p+1}:=\sum_{n=1}^\infty\frac{H_n\binom{2n}{n}}{4^n n^{p+1}}
    =-4\bibb{\ze(\bar{2},\{\hat{1}\}_p)
    +\sum_{l=0}^{p-1}\ze(\bar{1},\{\hat{1}\}_l,\hat{2},\{\hat{1}\}_{p-1-l})}\,.
\]
\end{corollary}

\begin{example}
Setting $p=0$ in Corollary \ref{Coro.S1p} gives
\[
\sum_{n=1}^\infty\frac{H_n\binom{2n}{n}}{4^n n}=2\ze(2)\,,
\]
in accordance with Alzer et al. \cite[Eq. (3.8)]{AlKarSri06} and Boyadzhiev \cite[Section 3]{Boyad12}. By further setting $p=1,2$ in Corollary \ref{Coro.S1p}, we obtain
\begin{align*}
\sum_{n=1}^\infty\frac{H_n\binom{2n}{n}}{4^n n^2}
    &=\frac{9}{2}\ze(3)-4\ze(2)\ln(2)\,,\\
\sum_{n=1}^\infty\frac{H_n\binom{2n}{n}}{4^n n^3}
    &=8\Lii_4\smbb{\frac{1}{2}}-\frac{13}{4}\ze(4)-2\ze(3)\ln(2)
        +2\ze(2)\ln^2(2)+\frac{1}{3}\ln^4(2)\,.
\end{align*}
\end{example}

When $m=2$, based on the value of $\ze_n^\star(1,1)$, we derive from Theorem \ref{Th.S*mp} the formula of the quadratic sums $\bms_{1^2,p}$.

\begin{corollary}\label{Coro.S11p}
For integer $p\geq 1$, the quadratic sums $\bms_{1^2,p}$ satisfy
\[
\bms_{1^2,p}=2\bms_{2,p}^\star-\bms_{2,p}\,.
\]
\end{corollary}

\begin{example}
As instances, we have
\begin{align*}
\sum_{n=1}^\infty\frac{H_n^2\binom{2n}{n}}{4^n n}
    &=\frac{21}{2}\ze(3)\,,\\
\sum_{n=1}^\infty\frac{H_n^2\binom{2n}{n}}{4^n n^2}
    &=32\Lii_4\smbb{\frac{1}{2}}-14\ze(4)+7\ze(3)\ln(2)-8\ze(2)\ln^2(2)+\frac{4}{3}\ln^4(2)\,.
\end{align*}
\end{example}

%%%%%%%%%%%%%%%%%%%%%%%%%%%%%%%%%%%%%%%%%%%%%%%%%%%%%%
%%%%%%%%%%%%%%%%%%%%%%%%%%%%%%%%%%%%%%%%%%%%%%%%%%%%%%
%%%%%%%%%%%%%%%%%%%%%%%%%%%%%%%%%%%%%%%%%%%%%%%%%%%%%%
%%%%%%%%%%%%%%%%%%%%%%%%%%%%%%%%%%%%%%%%%%%%%%%%%%%%%%
%%%%%%%%%%%%%%%%%%%%%%%%%%%%%%%%%%%%%%%%%%%%%%%%%%%%%%

\subsection{Quadratic sums $\bms_{1m,p}$ and cubic sums $\bms_{1^3,p}$}

The evaluations of $\bms_{1(m+1),p}$, for $m,p\geq 1$, can be deduced from the following theorem, which, together with Corollary \ref{Coro.S11p}, asserts that all the quadratic sums $\bms_{1m,p}$ can be expressed in terms of alternating MZVs.

\begin{theorem}\label{Th.S1mp}
For integers $m,p\geq1$, the quadratic sums $\bms_{1(m+1),p}$ satisfy
\begin{align*}
\bms_{1(m+1),p}:
&=\sum_{n=1}^\infty\frac{H_nH_n^{(m+1)}\binom{2n}{n}}{4^n n^p}\\
&=4\ze(\bar{1},\tilde{1},\{\hat{1}\}_{p-1},\hat{2},\{\hat{1}\}_{m-1})
    -4\ze(\tilde{2},\{\hat{1}\}_{p-1},\hat{2},\{\hat{1}\}_{m-1})
    +\ze(m+1)\bms_{1,p}\,.
\end{align*}
\end{theorem}

\begin{proof}
From \cite[Theorem 1]{ChenH16}, we obtain
\begin{equation}\label{gf.cb2}
\sum_{n=1}^\infty\frac{H_n\binom{2n}{n}}{4^n}t^n
    =\frac{2}{\sqrt{1-t}}\ln\smbb{\frac{1+\sqrt{1-t}}{2\sqrt{1-t}}}\,,
\end{equation}
and by Lemma \ref{Lem.int.Inm},  we have
\begin{align}
J(n,m):&=\int_0^1t^{n-1}\ln^m(t)\ln\smbb{\frac{1+t}{2t}}\uud t\nonumber\\
&=(-1)^mm!\sum_{j=1}^{m+1}\frac{(-1)^n}{n^{m+2-j}}(\ze(\bar{j})-\ze_n(\bar{j}))
    +\frac{(-1)^m(m+1)!}{n^{m+2}}\,.\label{int.Jnm}
\end{align}
Then by means of the transformation (\ref{zen.ze2n}), it can be found that
\begin{align*}
&\sum_{n=1}^\infty\frac{H_n(\ze(m+1)-H_n^{(m+1)})\binom{2n}{n}}{4^nn^p}
    =\frac{(-1)^m}{m!}\cdot2\int_0^1\frac{\ln^m(t)}{1-t}\uud t
    \underbrace{\frac{\ud t}{t}\cdots\frac{\ud t}{t}}_{p-1}
    \frac{\ln\smbb{\frac{1+\sqrt{1-t}}{2\sqrt{1-t}}}}{t\sqrt{1-t}}\uud t\\
&\quad=4\sum_{n=1}^{\infty}\ze_n(\{1\}_{p-1},2,\{1\}_{m-1})
    \cdot J(2n+1,0)\\
&\quad=-4\sum_{n=1}^{\infty}\sum_{\si_j\in\{\pm1\}}
    \ze_{n-1}\smbb{\cat_{i=1}^{p-1}\{\si_i\},2\si_p,\cat_{i=p+1}^{p+m-1}\{\si_i\}}
    \frac{1-(-1)^n}{n}\bibb{\sum_{k=n+1}^{\infty}\frac{(-1)^k}{k}-\frac{1}{n}}\\
&\quad=-4\ze(\bar{1},\tilde{1},\{\hat{1}\}_{p-1},\hat{2},\{\hat{1}\}_{m-1})
    +4\ze(\tilde{2},\{\hat{1}\}_{p-1},\hat{2},\{\hat{1}\}_{m-1})\,,
\end{align*}
for $m,p\geq 1$, which gives the desired result.
\end{proof}

Additionally, based on the value of $\ze_n^\star(1,1,1)$, we can compute the cubic sums $\bms_{1^3,p}$ by Theorems \ref{Th.Smp}, \ref{Th.S*mp} and \ref{Th.S1mp}:

\begin{corollary}\label{Coro.S111p}
For integer $p\geq 1$, the cubic sums $\bms_{1^3,p}$ satisfy
\[
\bms_{1^3,p}=6\bms_{3,p}^\star-3\bms_{12,p}-2\bms_{3,p}\,.
\]
\end{corollary}

\begin{example}
By further setting $p=1$ in Theorem \ref{Th.S1mp} and Corollary \ref{Coro.S111p} gives
\begin{align*}
&\sum_{n=1}^\infty\frac{H_nH_n^{(2)}\binom{2n}{n}}{4^n n}
    =-8\Lii_4\smbb{\frac{1}{2}}+\frac{49}{4}\ze(4)-7\ze(3)\ln(2)+2\ze(2)\ln^2(2)
    -\frac{1}{3}\ln^4(2)\,,\\
&\sum_{n=1}^\infty\frac{H_n^3\binom{2n}{n}}{4^n n}
    =40\Lii_4\smbb{\frac{1}{2}}+\frac{115}{4}\ze(4)+35\ze(3)\ln(2)
    -10\ze(2)\ln^2(2)+\frac{5}{3}\ln^4(2)\,.
\end{align*}
\end{example}

\begin{example}
Let $(x)_n$ be the \emph{rising factorial} (or the \emph{Pochhammer symbol}) defined by $(x)_0=1$ and $(x)_n=x(x+1)\cdots(x+n-1)$ for $n\geq 1$. Then $(\frac{1}{2})_n=4^{-n}n!\binom{2n}{n}$, and
\begin{align*}
&\sum_{n=1}^{\infty}\frac{(\frac{1}{2})_n}{n\cdot n!}H_{n-1}=\bms_{1,1}-\bms_2
    =\ze(2)+2\ln^2(2)\,,\\
&\sum_{n=1}^{\infty}\frac{(\frac{1}{2})_n}{n\cdot n!}\{H_{n-1}^2-H_{n-1}^{(2)}\}
    =4\ze(3)+4\ze(2)\ln(2)+\frac{8}{3}\ln^3(2)\,,\\
&\sum_{n=1}^{\infty}\frac{(\frac{1}{2})_n}{n\cdot n!}\{H_{n-1}^3-H_{n-1}^{(3)}\}\\
&\quad=-24\Lii_4\smbb{\frac{1}{2}}+\frac{207}{4}\ze(4)+15\ze(3)\ln(2)
    +18\ze(2)\ln^2(2)-\ln^4(2)\,,
\end{align*}
where the first two series can be found in the work of Srivastava and Choi \cite[p. 354]{SriChoi12}.
\hfill\qedsymbol
\end{example}

%%%%%%%%%%%%%%%%%%%%%%%%%%%%%%%%%%%%%%%%%%%%%%%%%%%%%%
%%%%%%%%%%%%%%%%%%%%%%%%%%%%%%%%%%%%%%%%%%%%%%%%%%%%%%
%%%%%%%%%%%%%%%%%%%%%%%%%%%%%%%%%%%%%%%%%%%%%%%%%%%%%%
%%%%%%%%%%%%%%%%%%%%%%%%%%%%%%%%%%%%%%%%%%%%%%%%%%%%%%
%%%%%%%%%%%%%%%%%%%%%%%%%%%%%%%%%%%%%%%%%%%%%%%%%%%%%%

\subsection{Quadratic sums $\bms^\star_{1m,p}$}\label{Sec.S*1mp}

Similarly, the explicit formula of $\bms_{1m,p}^\star$ can be established.

\begin{theorem}\label{Th.S*1mp}
For integers $m\geq 0$ and $p\geq 1$, the quadratic sums $\bms_{1m,p}^\star$ satisfy
\begin{align*}
&\bms_{1m,1}^\star:=\sum_{n=1}^\infty\frac{H_n\ze_n^\star(\{1\}_m)\binom{2n}{n}}{4^n n}
    =-2^{m+1}\sum_{k=1}^{m+1}\ze(\bar{k},\widetilde{m+2-k})+2^{m+1}(m+1)\ze(\widetilde{m+2})\,,\\
&\begin{aligned}
\bms_{1m,p+1}^\star:=\sum_{n=1}^\infty\frac{H_n\ze_n^\star(\{1\}_m)\binom{2n}{n}}{4^n n^{p+1}}
&=-2^{m+1}\sum_{\substack{|k|_{p+1}\leq w-1\\k_1,\ldots,k_{p+1}\geq1}}
    \ze\smbb{\bar{k}_{p+1},\widetilde{w-|k|_{p+1}},\cat_{i=1}^p\{\hat{k}_i\}}\nonumber\\
&\quad+2^{m+1}\sum_{\substack{|k|_p\leq w-2\\k_1,\ldots,k_p\geq1}}
    (w-1-|k|_p)\ze\smbb{\widetilde{w-|k|_p},\cat_{i=1}^p\{\hat{k}_i\}}\,,
\end{aligned}
\end{align*}
where $|k|_l:=k_1+k_2+\cdots+k_l$, and $w:=m+p+2$ is the weight of the corresponding sums.
\end{theorem}

\begin{proof}
When $p\geq 1$, using the integral (\ref{zn*1k}) and generating function (\ref{gf.cb2}), and then applying the transformation (\ref{zen.ze2n}) and the integral (\ref{int.Jnm}), we have
\begin{align*}
&\sum_{n=1}^\infty\frac{H_n\ze_n^\star(\{1\}_m)\binom{2n}{n}}{4^n n^{p+1}}
    =\frac{(-1)^m}{m!}\cdot 2\int_0^1\frac{\ln^m(1-t)}{t}\uud t
    \underbrace{\frac{\ud t}{t}\cdots\frac{\ud t}{t}}_{p-1}
    \frac{\ln\smbb{\frac{1+\sqrt{1-t}}{2\sqrt{1-t}}}}{t\sqrt{1-t}}\uud t\\
&\quad=(-1)^m\sum_{0\leq|k|_p\leq m}\frac{(-1)^{|k|_p}}{(m-|k|_p)!}
    2^{m+2-|k|_p}\sum_{n=1}^{\infty}\ze_{n-1}\smbb{\cat_{i=1}^p\{k_i+1\}}J(2n-1,m-|k|_p)\\
&\quad=-2^{m+1}\sum_{0\leq|k|_p\leq m}
    \sum_{n=1}^{\infty}\sum_{\si_i\in\{\pm1\}}\ze_{n-1}\smbb{\cat_{i=1}^p\{(k_i+1)\si_i\}}\\
&\quad\quad\times\bibb{\sum_{j=1}^{m+1-|k|_p}\frac{1-(-1)^n}{n^{m+2-|k|_p-j}}
    (\ze(\bar{j})-\ze_n(\bar{j}))-(m+1-|k|_p)\frac{1-(-1)^n}{n^{m+2-|k|_p}}}\,,
\end{align*}
where $k_i\geq 0$, for $i=1,2,\ldots,p$. Thus the formula of $\bms_{1m,p+1}^\star$ can be established. Next, since
\[
\sum_{n=1}^\infty\frac{H_n\ze_n^\star(\{1\}_m)\binom{2n}{n}}{4^n n}
    =\frac{(-1)^m}{m!}\cdot 2\int_0^1\frac{\ln^m(1-t)}{t}
    \frac{\ln\smbb{\frac{1+\sqrt{1-t}}{2\sqrt{1-t}}}}{\sqrt{1-t}}\uud t\,;
\]
then by computation, the formula of $\bms_{1m,1}^\star$ can also be obtained.
\end{proof}

Setting $m=0,1,2$ in Theorem \ref{Th.S*1mp} gives new expressions of the linear sums $\bms_{1,p+1}$, the quadratic sums $\bms_{1^2,p}$ and the cubic sums $\bms_{1^3,p}$, which are simpler than those presented before.

\begin{corollary}\label{Coro.S1p.11p.111p}
For integer $p\geq 0$, the linear sums $\bms_{1,p+1}$ satisfy
\[
\bms_{1,p+1}:=\sum_{n=1}^\infty\frac{H_n\binom{2n}{n}}{4^n n^{p+1}}
    =-2\ze(\bar{1},\tilde{1},\{\hat{1}\}_p)+2\ze(\tilde{2},\{\hat{1}\}_p)\,.
\]
For integer $p\geq 1$, the quadratic sums $\bms_{1^2,p}$ and the cubic sums $\bms_{1^3,p}$ satisfy
\[
\bms_{1^2,p}=\bms_{1^2,p}^\star\,,\quad
    \bms_{1^3,p}=2\bms_{12,p}^\star-\bms_{12,p}\,.
\]
\end{corollary}

%%%%%%%%%%%%%%%%%%%%%%%%%%%%%%%%%%%%%%%%%%%%%%%%%%%%%%
%%%%%%%%%%%%%%%%%%%%%%%%%%%%%%%%%%%%%%%%%%%%%%%%%%%%%%
%%%%%%%%%%%%%%%%%%%%%%%%%%%%%%%%%%%%%%%%%%%%%%%%%%%%%%
%%%%%%%%%%%%%%%%%%%%%%%%%%%%%%%%%%%%%%%%%%%%%%%%%%%%%%
%%%%%%%%%%%%%%%%%%%%%%%%%%%%%%%%%%%%%%%%%%%%%%%%%%%%%%

\section{Evaluations via contour integration}\label{Sec.Con.Int}

%%%%%%%%%%%%%%%%%%%%%%%%%%%%%%%%%%%%%%%%%%%%%%%%%%%%%%
%%%%%%%%%%%%%%%%%%%%%%%%%%%%%%%%%%%%%%%%%%%%%%%%%%%%%%
%%%%%%%%%%%%%%%%%%%%%%%%%%%%%%%%%%%%%%%%%%%%%%%%%%%%%%
%%%%%%%%%%%%%%%%%%%%%%%%%%%%%%%%%%%%%%%%%%%%%%%%%%%%%%
%%%%%%%%%%%%%%%%%%%%%%%%%%%%%%%%%%%%%%%%%%%%%%%%%%%%%%

\subsection{Lemmas}\label{Sec.Lem}

The contour integration is an efficient method to evaluate infinite series by reducing them to a finite number of residue computations. Flajolet and Salvy \cite{FlSa98} used this method to compute the classical Euler sums, and found many interesting results. Now, let us consider the Euler-Ap\'{e}ry-type series.

\begin{lemma}[\citu{FlSa98}{Lemma 2.1}]\label{Lem.Res}
Let $\xi(z)$ be a kernel function and let $r(z)$ be a rational function which is $O(z^{-2})$ at infinity. Then
\begin{equation}\label{Cau.Lind}
\sum_{\al\in O}{\rm Res}(r(z)\xi(z),\al)
    +\sum_{\be\in S}{\rm Res}(r(z)\xi(z),\be)=0\,,
\end{equation}
where $S$ is the set of poles of $r(z)$ and $O$ is the set of poles of $\xi(z)$ that are not poles of $r(z)$. Here ${\rm Res}(h(z),\la)$ denotes the residue of $h(z)$ at $z=\la$, and the kernel function $\xi(z)$ is meromorphic in the whole complex plane and satisfies $\xi(z)=o(z)$ over an infinite collection of circles $|z|=\rho_k$ with $\rho_k\to+\infty$.
\end{lemma}

It is clear that the formula (\ref{Cau.Lind}) is also true if $r(z)\xi(z)=o(z^{-\al})$ $(\al>1)$ over an infinite collection of circles $|z|=\rho_k$ with $\rho_k\to+\infty$.

Flajolet and Salvy \cite{FlSa98} gave the expansions of some basic kernel functions. Here, we need to cite the five formulas:
\begin{align*}
&\pi\cot(\pi z)\stackrel{z\to n}{=}
    \frac{1}{z-n}-2\sum_{k=1}^\infty\ze(2k)(z-n)^{2k-1}\,,\\
&\psi(-z)+\ga\stackrel{z\to n}{=}
    \frac{1}{z-n}+H_n+\sum_{k=1}^\infty\{(-1)^kH_n^{(k+1)}-\ze(k+1)\}(z-n)^k\,,
    \quad\text{if }  n\geq 0\,,\\
&\psi(-z)+\ga\stackrel{z\to -n}{=}
    H_{n-1}+\sum_{k=1}^\infty\{H_{n-1}^{(k+1)}-\ze(k+1)\}(z+n)^k\,,
    \quad\text{if }  n>0\,,\\
&\begin{aligned}
\frac{\psi^{(p-1)}(-z)}{(p-1)!}
&\stackrel{z\to n}{=}
    \frac{1}{(z-n)^p}\smbb{1+(-1)^p\sum_{i\geq p}\binom{i-1}{p-1}
    \{\ze(i)+(-1)^iH_n^{(i)}\}(z-n)^i}\,,\\
&\quad\quad\text{if } n\geq 0\,,\ p>1\,,\\
\end{aligned}\\
&\frac{\psi^{(p-1)}(-z)}{(p-1)!}\stackrel{z\to -n}{=}
    (-1)^p\sum_{i\geq 0}\binom{p-1+i}{p-1}\{\ze(p+i)-H_{n-1}^{(p+i)}\}(z+n)^i\,,
    \quad\text{if } n>0,\ p>1\,.
\end{align*}

\begin{lemma}\label{Lem.CD}
For $|z|<1$, the following identities hold:
\[
\Gamma(z+1)\uue^{\ga z}=\sum_{n=0}^\infty C_n\frac{z^n}{n!}\,,\quad
\{\Gamma(z+1)\uue^{\ga z}\}^{-1}=\sum_{n=0}^\infty D_n\frac{z^n}{n!}\,,
\]
where $\ga:=\lim_{n\to\infty}(H_n-\ln n)$ is the Euler-Mascheroni constant, and
\begin{align*}
&C_n:=Y_n(0,1!\ze(2),-2!\ze(3),\ldots,(-1)^{n}(n-1)!\ze(n))\,,\\
&D_n:=Y_n(0,-1!\ze(2),2!\ze(3),\ldots,(-1)^{n-1}(n-1)!\ze(n))\,.
\end{align*}
\end{lemma}

\begin{proof}
The lemma follows from the known expansion
\[
\frac{1}{\Gamma(z)}=z\exp\bibb{\ga z-\sum_{k=2}^{\infty}\frac{(-1)^k\ze(k)}{k}z^k}
\]
and the generating function (\ref{cBell.gf}) of the Bell polynomials.
\end{proof}

According to Lemma \ref{Lem.CD}, it is clear that $C_k$ and $D_k$ are rational linear combinations of products of zeta values. In particular, by (\ref{cBell.rec}), we have
$(C_k)_{k\in\mathbb{N}_0}=(1,0,\ze(2),-2\ze(3),\frac{27}{2}\ze(4),\ldots)$ and $(D_k)_{k\in\mathbb{N}_0}=(1,0,-\ze(2),2\ze(3),\frac{3}{2}\ze(4),\ldots)$.
Furthermore, by the relations
\begin{align*}
&\frac{\Gamma(z-n)}{\Gamma(z+1)}
    =\prod_{k=0}^n\frac{1}{z-k}
    =\frac{(-1)^n}{n!}\frac{1}{z}\sum_{m=0}^\infty\ze_n^\star(\{1\}_m)z^m\,,\\
&\frac{\Gamma(z+1)}{\Gamma(z-n)}
    =\prod_{k=0}^n(z-k)
    =(-1)^nn!z\sum_{m=0}^\infty(-1)^m\ze_n(\{1\}_m)z^m\,,
\end{align*}
for $n\geq 0$, the next lemma can be established.

\begin{lemma}\label{Lem.AB}
For nonnegative integer $n$, when $z\to-n$, we have
\begin{align*}
&\Gamma(z)\uue^{\ga(z-1)}
    =\frac{(-1)^n}{n!}\uue^{-\ga(n+1)}\sum_{k=0}^\infty A_k(n)(z+n)^{k-1}\,,\\
&\frac{1}{\Gamma(z)\uue^{\ga(z-1)}}
    =(-1)^nn!\uue^{\ga(n+1)}\sum_{k=0}^\infty B_k(n)(z+n)^{k+1}\,,
\end{align*}
where
\[
A_k(n):=\sum_{\substack{k_1+k_2=k\\k_1,k_2\geq 0}}
    \ze_n^\star(\{1\}_{k_1})\frac{C_{k_2}}{k_2!}\,,\quad
B_k(n):=\sum_{\substack{k_1+k_2=k\\k_1,k_2\geq 0}}
    (-1)^{k_1}\ze_n(\{1\}_{k_1})\frac{D_{k_2}}{k_2!}\,.
\]
\end{lemma}

%%%%%%%%%%%%%%%%%%%%%%%%%%%%%%%%%%%%%%%%%%%%%%%%%%%%%%
%%%%%%%%%%%%%%%%%%%%%%%%%%%%%%%%%%%%%%%%%%%%%%%%%%%%%%
%%%%%%%%%%%%%%%%%%%%%%%%%%%%%%%%%%%%%%%%%%%%%%%%%%%%%%
%%%%%%%%%%%%%%%%%%%%%%%%%%%%%%%%%%%%%%%%%%%%%%%%%%%%%%
%%%%%%%%%%%%%%%%%%%%%%%%%%%%%%%%%%%%%%%%%%%%%%%%%%%%%%

\subsection{Single sums $\bms_{q}$ and $\tilde{\bms}_{q}$}\label{Sec.Sq.tSq}

Next, we use the contour integral and residue theorem to evaluate some Euler-Ap\'{e}ry-type series. Firstly, denote
\begin{equation}
G(\vec{k}):=G(k_1,k_2,k_3,k_4)
    =\frac{2^{k_3+k_4}\ln^{k_4}(2)C_{k_1}C_{k_2}D_{k_3}}{k_1!k_2!k_3!k_4!}\,.
\end{equation}
Then we have a new closed form expression for the single sums $\bms_{q}$.

\begin{theorem}\label{Th.Sq}
For integer $q\geq 2$, the single sums $\bms_{q-1}$ are reducible to $\ln(2)$ and zeta values. In particular, we have
\[
\bms_{q-1}:=\sum_{n=1}^\infty\frac{\binom{2n}{n}}{4^n n^{q-1}}
    =(-1)^q\sum_{|k|_4=q-1}G(k_1,k_2,k_3,k_4)\,,
\]
where $k_1,\ldots,k_4\geq 0$.
\end{theorem}

\begin{proof}
Consider the function
\[
F_0(z):=\frac{\Gamma^2(z+1)}{\Gamma(2z+1)}\frac{4^z}{z^q}\,.
\]
With the help of the Legendre duplication formula
\[
\Gamma(z)\Gamma\smbb{z+\frac{1}{2}}=\sqrt{\pi}\cdot 2^{1-2z}\Gamma(2z)
\]
and the asymptotic expansion for the ratio of two gamma functions
\[
\frac{\Gamma(z+a)}{\Gamma(z+b)}
    =z^{a-b}\smbb{1+O\smbb{\frac{1}{z}}}\,,
    \quad\text{for } |\arg(z)|\leq\pi-\vep\,,\ \vep>0\,,\ |z|\to\infty
\]
(see \cite[Sections 2.3 and 2.11]{Luke69.1}), we have
\[
F_0(z)=\frac{\sqrt{\pi}}{z^{q-1/2}}\smbb{1+O\smbb{\frac{1}{z}}}\,,\quad |z|\to\infty\,,
\]
which implies that the integral $\oint_{(\infty)}F_0(z)\uud z=0$, where $\oint_{(\infty)}$ denotes integration along large circles, that is, the limit of integrals $\oint_{|s|=\rho_k}$. On the other hand, the function $F_0(z)$ has poles only at all non-positive integers. For a negative integer $-n$, by Lemma \ref{Lem.AB}, if $z\to -n$, we have
\[
F_0(z)=-\frac{4^z}{z^q}\frac{(2n-1)!}{(n-1)!^2}
    \sum_{k_1,k_2,k_3\geq0}2^{k_3+1}A_{k_1}(n-1)A_{k_2}(n-1)B_{k_3}(2n-1)
    (z+n)^{k_1+k_2+k_3-1}\,.
\]
Hence, the residue is
\[
{\rm Res}(F_0(z),-n)=(-1)^{q+1}\frac{\binom{2n}{n}}{4^nn^{q-1}}\,,
    \quad\text{for }n=1,2,\ldots\,.
\]
Similarly, if $z\to 0$, we have
\[
F_0(z)=\frac{4^z}{z^q}\sum_{k_1,k_2,k_3\geq0}
    \frac{2^{k_3}C_{k_1}C_{k_2}D_{k_3}}{k_1!k_2!k_3!}z^{k_1+k_2+k_3}\,,
\]
and the residue of the pole of order $q$ at 0 is
\[
{\rm Res}(F_0(z),0)=\sum_{|k|_4=q-1}G(k_1,k_2,k_3,k_4)\,.
\]
Summing these two contributions gives the statement of the theorem.
\end{proof}

Next, we show the expression of the single sums $\tilde{\bms}_{q}$.

\begin{theorem}\label{Th.tSq}
For integer $q\geq 2$, the single sums $\tilde{\bms}_{q}$ satisfy
\[
\tilde{\bms}_{q}:=\sum_{n=1}^\infty\frac{4^n}{n^q\binom{2n}{n}}
    =(-1)^q\bms_{1,q-1}+\sum_{|k|_5=q-2}G(k_1,k_2,k_3,k_4)\ze(k_5+2)\,,
\]
where $k_1,\ldots,k_5\geq 0$.
\end{theorem}

\begin{proof}
Similarly to Theorem \ref{Th.Sq}, consider the function
$F_1(z):=(\psi(-z)+\ga)F_0(z)$, which has poles only at all integers. At a positive integer $n$, the pole is simple and the residue is
\[
{\rm Res}(F_1(z),n)=\frac{4^n}{n^q\binom{2n}{n}}\,.
\]
At a negative integer $-n$, the pole is simple and the residue is
\[
{\rm Res}(F_1(z),-n)=(-1)^{q+1}\frac{H_{n-1}\binom{2n}{n}}{4^nn^{q-1}}\,.
\]
Finally, the residue of the pole of order $q+1$ at $0$ is found to be
\[
{\rm Res}(F_1(z),0)=\sum_{|k|_4=q}G(k_1,k_2,k_3,k_4)
    -\sum_{|k|_5=q-2}G(k_1,k_2,k_3,k_4)\ze(k_5+2)\,.
\]
By Lemma \ref{Lem.Res} and Theorem \ref{Th.Sq}, summing these three contributions yields the desired result.
\end{proof}

\begin{example}
By Corollary \ref{Coro.S1p.11p.111p}, the sums $\tilde{\bms}_{q}$ can be evaluated by some known constants:
\begin{align*}
\sum_{n=1}^\infty\frac{4^n}{n^2\binom{2n}{n}}&=3\ze(2)\,,\\
\sum_{n=1}^\infty\frac{4^n}{n^3\binom{2n}{n}}&=-\frac{7}{2}\ze(3)+6\ze(2)\ln(2)\,,\\
\sum_{n=1}^\infty\frac{4^n}{n^4\binom{2n}{n}}
    &=8\Lii_4\smbb{\frac{1}{2}}-\frac{19}{4}\ze(4)+4\ze(2)\ln^2(2)+\frac{1}{3}\ln^4(2)\,.
\end{align*}
The evaluation of $\tilde{\bms}_{2}$ was presented by Sprugnoli \cite[Theorem 3.1]{Spru06}. The value of $\tilde{\bms}_{3}$ was given in Coppo and Candelpergher \cite[Example 8]{CopCan15} and Zucker \cite[Eq. (2.12)]{Zuck85}, and this formula is also equivalent to a result known to Ramanujan \cite[p. 269]{Berndt85.1}. Additionally, the formula of $\tilde{\bms}_{4}$ is equivalent to the following one on the \emph{Ramanujan constant} $G(1)$:
\[
G(1)=\frac{7}{8}\ze(3)\ln(2)-\frac{1}{384}\pi^4-\frac{1}{8}\pi^2\ln^2(2)
    +2\sum_{n=1}^{\infty}\frac{2^{2n}}{(2n)^4\binom{2n}{n}}
\]
(see \cite[Proposition 10]{CopCan15}), where
\begin{align*}
G(1):&=\sum_{n=1}^{\infty}\frac{O_n}{(2n)^3}=\frac{35}{64}\ze(4)
    -\frac{1}{2}\sum_{n=1}^{\infty}(-1)^{n-1}\frac{H_n}{n^3}\\
&=\Lii_4\smbb{\frac{1}{2}}-\frac{53}{64}\ze(4)+\frac{7}{8}\ze(3)\ln(2)
    -\frac{1}{4}\ze(2)\ln^2(2)+\frac{1}{24}\ln^4(2)\,.
\end{align*}
The readers are also referred to \cite[p. 257]{Berndt85.1} and \cite{Sit87} for more details.
\hfill\qedsymbol
\end{example}

Using Eq. (\ref{int.tn.lnm.1}) and the series expansion
\[
\arcsin^2(z)=\frac{1}{2}\sum_{n=1}^\infty\frac{{(2z)}^{2n}}{n^2\binom{2n}{n}}
\]
(see \cite[Section 1.5]{BBCGLM07}), we obtain the integral representation:
\begin{equation}
\tilde{\bms}_{q+2}=\sum_{n=1}^\infty\frac{4^n}{n^{q+2}\binom{2n}{n}}
    =\frac{(-1)^q}{q!}2^{q+2}\int_0^{\pi/2}t\ln^q(\sin t)\uud t\,,
\end{equation}
where the evaluation of the log-sine integrals on the right was established by Orr \cite[Theorem 2.2]{Orr19} most recently. Note that this integral representation is in fact the simplification of Coppo and Candelpergher's result \cite[Proposition 8]{CopCan15}. Moreover, from another point of view, Theorem \ref{Th.tSq} also gives a new formula of this kind of log-sine integrals in terms of $\ln(2)$, zeta values and alternating MZVs.

Recently, Chen \cite[Section 5]{ChenKW19} also established an explicit formula for the sums $\tilde{\bms}_{q}$:
\[
\tilde{\bms}_{q}=\sum_{n=1}^\infty\frac{4^n}{n^q\binom{2n}{n}}
    =2^q t(2,\{1\}_{q-2})=2\ze(\tilde{2},\{\tilde{1}\}_{q-2})\,.
\]
Hence, combining it with Corollary \ref{Coro.S1p.11p.111p} and Theorem \ref{Th.Sq}, we obtain the following relation
\begin{align*}
&\ze(\tilde{2},\{\tilde{1}\}_{q-2})
    -(-1)^q\ze(\tilde{2},\{\hat{1}\}_{q-2})
    +(-1)^q\ze(\bar{1},\tilde{1},\{\hat{1}\}_{q-2})\\
&\quad=\frac{1}{2}\sum_{|k|_5=q-2}G(k_1,k_2,k_3,k_4)\ze(k_5+2)\,.
\end{align*}
%%%%%%%%%%%%%%%%%%%%%%%%%%%%%%%%%%%%%%%%%%%%%%%%%%%%%%
%%%%%%%%%%%%%%%%%%%%%%%%%%%%%%%%%%%%%%%%%%%%%%%%%%%%%%
%%%%%%%%%%%%%%%%%%%%%%%%%%%%%%%%%%%%%%%%%%%%%%%%%%%%%%
%%%%%%%%%%%%%%%%%%%%%%%%%%%%%%%%%%%%%%%%%%%%%%%%%%%%%%
%%%%%%%%%%%%%%%%%%%%%%%%%%%%%%%%%%%%%%%%%%%%%%%%%%%%%%

\subsection{Linear sums $\tilde{\bms}_{1,q}$ and related sums}\label{Sec.tS1q}

By contour integrals and residue computation, more results can be obtained.

\begin{theorem}\label{Th.tS1q.tT1q}
For integer $q\geq 2$, the linear sums $\tilde{\bms}_{1,q}$ and $\tilde{\bmt}_{1,q}$ satisfy
\begin{align*}
\tilde{\bms}_{1,q}:&=\sum_{n=1}^{\infty}\frac{4^nH_n}{n^q\binom{2n}{n}}
    =\tilde{\bms}_{q+1}+\frac{(-1)^q}{2}\{\bms_{1^2,q-1}-\ze(2)\bms_{q-1}+\bms_{2,q-1}\}\\
&\quad+\frac{1}{2}\sum_{|k|_5=q-1}G(\vec{k})(k_5+1)\ze(k_5+2)
    -\frac{1}{2}\sum_{|k|_6=q-3}G(\vec{k})\ze(k_5+2)\ze(k_6+2)\,,
\end{align*}
and
\begin{align*}
\tilde{\bmt}_{1,q}:&=\sum_{n=1}^{\infty}\frac{4^nH_{2n}}{n^q\binom{2n}{n}}
    =\smbb{1-\frac{q}{2}}\tilde{\bms}_{q+1}+\ln(2)\tilde{\bms}_q
    +\frac{(-1)^q}{2}\{\bms_{1^2,q-1}-2\ze(2)\bms_{q-1}+2\bms_{2,q-1}\}\\
&\quad+\sum_{|k|_5=q-1}G(\vec{k})(k_5+1)\ze(k_5+2)
    -\frac{1}{2}\sum_{|k|_6=q-3}G(\vec{k})\ze(k_5+2)\ze(k_6+2)\,,
\end{align*}
where $k_1,\ldots,k_6\geq 0$ and $\vec{k}=(k_1,k_2,k_3,k_4)$.
\end{theorem}

\begin{proof}
Setting $F_2(z):=(\psi(-z)+\ga)^2F_0(z)$ and $F_3(z):=\psi^{(1)}(-z)F_0(z)$, respectively, and performing the residue computation, we obtain the expressions of the sums
\[
\sum_{n=1}^{\infty}\frac{4^n(2H_n-H_{2n})}{n^q\binom{2n}{n}}\quad\text{and}\quad
\sum_{n=1}^{\infty}\frac{4^n(H_n-H_{2n})}{n^q\binom{2n}{n}}\,,
\]
respectively, which further give the results in this theorem.
\end{proof}

\begin{example}
By specifying the parameter $q$ in Theorem \ref{Th.tS1q.tT1q}, we have
\begin{align*}
&\sum_{n=1}^{\infty}\frac{4^nH_n}{n^2\binom{2n}{n}}
    =\frac{7}{2}\ze(3)+6\ze(2)\ln(2)\,,\\
&\sum_{n=1}^{\infty}\frac{4^nH_n}{n^3\binom{2n}{n}}
    =-8\Lii_4\smbb{\frac{1}{2}}+\ze(4)+8\ze(2)\ln^2(2)-\frac{1}{3}\ln^4(2)\,,
\end{align*}
and
\begin{align*}
\sum_{n=1}^{\infty}\frac{4^nH_{2n}}{n^2\binom{2n}{n}}
    &=\frac{35}{4}\ze(3)+3\ze(2)\ln(2)\,,\\
\sum_{n=1}^{\infty}\frac{4^nH_{2n}}{n^3\binom{2n}{n}}
    &=-20\Lii_4\smbb{\frac{1}{2}}+\frac{65}{8}\ze(4)+8\ze(2)\ln^2(2)-\frac{5}{6}\ln^4(2)\,.
\end{align*}
In addition, the next two series for the Ap\'{e}ry's constant $\ze(3)$ can be obtained:
\[
\sum_{n=1}^{\infty}\frac{4^nH_{n-1}}{n^2\binom{2n}{n}}
    =\tilde{\bms}_{1,2}-\tilde{\bms}_3=7\ze(3)\,,\quad
\sum_{n=1}^{\infty}\frac{4^nH_{2n-1}}{n^2\binom{2n}{n}}
    =\tilde{\bmt}_{1,2}-\frac{1}{2}\tilde{\bms}_3=\frac{21}{2}\ze(3)\,,
\]
which was given by Sun \cite[Remark 5.2]{Sun15.ANSR} by using Mathematica, and the first one can also be found in a recent paper due to Chen \cite[Section 5]{ChenKW19}.\hfill\qedsymbol
\end{example}

\begin{theorem}\label{Th.T1q}
For integers $q\geq 1$, the linear sums $\bmt_{1,q}$ satisfy
\begin{align*}
\bmt_{1,q}:&=\sum_{n=1}^{\infty}\frac{H_{2n}\binom{2n}{n}}{4^nn^q}
    =\ln(2)\bms_q+\frac{q}{2}\bms_{q+1}+\bms_{1,q}+\frac{(-1)^q}{2}\tilde{\bms}_{q+1}
    -\frac{1}{2}\sum_{|k|_4=q+1}H(\vec{k})\\
&\quad+2\sum_{|k|_4+2k_5=q-1}H(\vec{k})\ze(2k_5+2)
    -2\sum_{|k|_4+2k_5+2k_6=q-3}H(\vec{k})\ze(2k_5+2)\ze(2k_6+2)\,,
\end{align*}
where $k_1,\ldots,k_6\geq 0$, and
\[
H(\vec{k}):=H(k_1,k_2,k_3,k_4)
    =\frac{(-1)^{k_4}2^{k_1+k_4}\ln^{k_4}(2)C_{k_1}D_{k_2}D_{k_3}}{k_1!k_2!k_3!k_4!}\,.
\]
\end{theorem}

\begin{proof}
In this case, let
\[
F_4(z)=\frac{\pi^2\cot^2(\pi z)}{4^zz^q}\frac{\Gamma(2z+1)}{\Gamma^2(z+1)}\,.
\]
Based on residue computation and Lemmas \ref{Lem.Res} -- \ref{Lem.AB}, we obtain this formula.
\end{proof}

\begin{example}
From Theorem \ref{Th.T1q}, we have
\begin{align*}
\sum_{n=1}^{\infty}\frac{H_{2n}\binom{2n}{n}}{4^nn}&=\frac{5}{2}\ze(2)\,,\\
\sum_{n=1}^{\infty}\frac{H_{2n}\binom{2n}{n}}{4^nn^2}
    &=\frac{23}{4}\ze(3)-5\ze(2)\ln(2)\,,\\
\sum_{n=1}^{\infty}\frac{H_{2n}\binom{2n}{n}}{4^nn^3}
    &=4\Lii_4\smbb{\frac{1}{2}}+\frac{17}{8}\ze(4)-8\ze(3)\ln(2)
        +4\ze(2)\ln^2(2)+\frac{1}{6}\ln^4(2)\,.
\end{align*}
Using these evaluations, the following series is easy to compute:
\[
\sum_{n=1}^{\infty}\frac{(H_{2n-1}-H_n)\binom{2n}{n}}{4^nn}
    =\bmt_{1,1}-\frac{1}{2}\bms_2-\bms_{1,1}=\ln^2(2)\,,
\]
which is related to a problem proposed by Knuth \cite{Knuth15} in \emph{the American Mathematical Monthly} in 2015, and was discovered by Chen \cite[Section 3]{ChenH16}.\hfill\qedsymbol
\end{example}

Finally, let $O_n=\sum_{k=1}^n\frac{1}{2k-1}$ be the odd harmonic numbers; then $O_n=H_{2n}-\frac{1}{2}H_n$, and the next corollary holds.

\begin{corollary}
The linear sums $\bmu_{1,q}$ and $\tilde{\bmu}_{1,q}$ satisfy
\begin{align*}
&\bmu_{1,q}:=\sum_{n=1}^{\infty}\frac{O_n\binom{2n}{n}}{4^nn^q}=\bmt_{1,q}-\frac{1}{2}\bms_{1,q}
    \,,\quad\text{for }q\geq 1\,,\\
&\tilde{\bmu}_{1,q}:=\sum_{n=1}^{\infty}\frac{4^nO_n}{n^q\binom{2n}{n}}
    =\tilde{\bmt}_{1,q}-\frac{1}{2}\tilde{\bms}_{1,q}
    \,,\quad\text{for }q\geq 2\,.
\end{align*}
\end{corollary}

\begin{example}
From this corollary, we have
\begin{align*}
&\sum_{n=1}^{\infty}\frac{O_n\binom{2n}{n}}{4^nn}=\frac{3}{2}\ze(2)\,,\\
&\sum_{n=1}^{\infty}\frac{O_n\binom{2n}{n}}{4^nn^2}=\frac{7}{2}\ze(3)-3\ze(2)\ln(2)\,,\\
&\sum_{n=1}^{\infty}\frac{O_n\binom{2n}{n}}{4^nn^3}
    =\frac{15}{4}\ze(4)-7\ze(3)\ln(2)+3\ze(2)\ln^2(2)\,,
\end{align*}
where $\bmu_{1,1}$ was given by Alzer et al. \cite[Eq. (2.5)]{AlKarSri06} in a different form. Similarly, we have
\begin{align*}
\sum_{n=1}^{\infty}\frac{4^nO_n}{n^2\binom{2n}{n}}&=7\ze(3)\,,\\
\sum_{n=1}^{\infty}\frac{4^nO_n}{n^3\binom{2n}{n}}
    &=-16\Lii_4\smbb{\frac{1}{2}}+\frac{61}{8}\ze(4)+4\ze(2)\ln^2(2)-\frac{2}{3}\ln^4(2)\,.
\end{align*}
Note that the evaluation of $\tilde{\bmu}_{1,2}$ was presented in \cite[Section 5]{ChenKW19} and \cite[Eq. (9)]{CopCan15} recently, and the value of $\tilde{\bmu}_{1,3}$ is equivalent to another formula on the Ramanujan constant $G(1)$ given in \cite[Proposition 11]{CopCan15}.\hfill\qedsymbol
\end{example}

%%%%%%%%%%%%%%%%%%%%%%%%%%%%%%%%%%%%%%%%%%%%%%%%%%%%%%
%%%%%%%%%%%%%%%%%%%%%%%%%%%%%%%%%%%%%%%%%%%%%%%%%%%%%%
%%%%%%%%%%%%%%%%%%%%%%%%%%%%%%%%%%%%%%%%%%%%%%%%%%%%%%
%%%%%%%%%%%%%%%%%%%%%%%%%%%%%%%%%%%%%%%%%%%%%%%%%%%%%%
%%%%%%%%%%%%%%%%%%%%%%%%%%%%%%%%%%%%%%%%%%%%%%%%%%%%%%

\section{Further remarks}\label{Sec.Remark}

In this paper, we use iterated integrals and contour integrals to evaluate some Euler-Ap\'{e}ry-type series. Based on the results, for integers $m,p\geq 1$ and $q\geq 2$, the Euler-Ap\'{e}ry-type series
\begin{equation}\label{EAS.stu}
\bms_p\,,\ \bms_{m,p}\,,\ \bms_{1m,p}\,,\ \bms_{1^3,p}\,, \bmt_{1,p}\,,\ \bmu_{1,p}
\end{equation}
and
\begin{equation}\label{EAS.tstu}
\tilde{\bms}_{q}\,,\ \tilde{\bms}_{1,q}\,,\ \tilde{\bmt}_{1,q}\,,\ \tilde{\bmu}_{1,q}
\end{equation}
are expressible in terms of (alternating) MZVs. In particular, $\bms_p\in\mathbb{Q}[\ln(2),\ze(2),\ze(3),\ldots]$, that is, the single sums $\bms_p$ are reducible to $\ln(2)$ and zeta values. Moreover, we conjecture that all the Euler-Ap\'{e}ry-type series in (\ref{EAS.def}) can be expressed in terms of (alternating) MZVs.

In Sections \ref{Sec.Int.MZV} and \ref{Sec.Con.Int}, by specifying the parameters, we present some special Euler-Ap\'{e}ry-type series of weight $w\leq 4$. To evaluate these series automatically, we develop the Maple program \texttt{EASum} based on the corresponding explicit formulas and the multiple zeta value data mine \cite{BBV2010}. The readers can download it from \cite{WangXu19.MP.ESEA}. For example, using the commands
$\texttt{EASum([1,4,s])}$ and $\texttt{EASum([1,4,ts])}$, we obtain
\begin{align*}
\sum_{n=1}^\infty\frac{H_n\binom{2n}{n}}{4^n n^4}
    &=16\Lii_5\smbb{\frac{1}{2}}-\frac{23}{8}\ze(5)
        +\frac{13}{2}\ze(4)\ln(2)-5\ze(3)\ze(2)+2\ze(3)\ln^2(2)\\
    &\quad-\frac{4}{3}\ze(2)\ln^3(2)-\frac{2}{15}\ln^5(2)
\end{align*}
and
\[
\sum_{n=1}^{\infty}\frac{4^nH_n}{n^4\binom{2n}{n}}
    =16\Lii_5\smbb{\frac{1}{2}}-\frac{31}{2}\ze(5)+2\ze(4)\ln(2)
    +3\ze(3)\ze(2)+\frac{16}{3}\ze(2)\ln^3(2)-\frac{2}{15}\ln^5(2)\,,
\]
respectively. Similarly, when the last parameter in the command is replaced by ``\texttt{t}'', ``\texttt{tt}'', ``\texttt{u}'' and ``\texttt{tu}'', we have
\begin{align*}
&\begin{aligned}
\sum_{n=1}^{\infty}\frac{H_{2n}\binom{2n}{n}}{4^nn^4}
    &=8\Lii_5\smbb{\frac{1}{2}}+\frac{225}{16}\ze(5)-\frac{17}{4}\ze(4)\ln(2)+8\ze(3)\ln^2(2)\\
    &\quad-\frac{8}{3}\ze(2)\ln^3(2)-9\ze(3)\ze(2)-\frac{1}{15}\ln^5(2)\,,
\end{aligned}\\
&\begin{aligned}
\sum_{n=1}^{\infty}\frac{4^nH_{2n}}{n^4\binom{2n}{n}}
    &=40\Lii_5\smbb{\frac{1}{2}}-\frac{217}{16}\ze(5)-9\ze(3)\ze(2)
        +\frac{65}{4}\ze(4)\ln(2)\\
    &\quad+\frac{16}{3}\ze(2)\ln^3(2)-\frac{1}{3}\ln^5(2)\,,
\end{aligned}\\
&\sum_{n=1}^{\infty}\frac{O_n\binom{2n}{n}}{4^nn^4}
    =\frac{31}{2}\ze(5)-\frac{15}{2}\ze(4)\ln(2)-\frac{13}{2}\ze(3)\ze(2)
    +7\ze(3)\ln^2(2)-2\ze(2)\ln^3(2)\,,\\
&\begin{aligned}
\sum_{n=1}^{\infty}\frac{4^nO_n}{n^4\binom{2n}{n}}
    &=32\Lii_5\smbb{\frac{1}{2}}-\frac{93}{16}\ze(5)+\frac{61}{4}\ze(4)\ln(2)
    -\frac{21}{2}\ze(3)\ze(2)\\
    &\quad+\frac{8}{3}\ze(2)\ln^3(2)-\frac{4}{15}\ln^5(2)\,.
\end{aligned}
\end{align*}
In the attachment to \cite{WangXu19.MP.ESEA}, we present all the evaluations of the Euler-Ap\'{e}ry-type series in (\ref{EAS.stu}) and (\ref{EAS.tstu}) and of weight $w\leq 6$.

%%%%%%%%%%%%%%%%%%%%%%%%%%%%%%%%%%%%%%%%%%%%%%%%%%%%%%
%%%%%%%%%%%%%%%%%%%%%%%%%%%%%%%%%%%%%%%%%%%%%%%%%%%%%%
%%%%%%%%%%%%%%%%%%%%%%%%%%%%%%%%%%%%%%%%%%%%%%%%%%%%%%
%%%%%%%%%%%%%%%%%%%%%%%%%%%%%%%%%%%%%%%%%%%%%%%%%%%%%%
%%%%%%%%%%%%%%%%%%%%%%%%%%%%%%%%%%%%%%%%%%%%%%%%%%%%%%
\section*{Acknowledgments}

The authors would like to express their deep gratitude to Professor Masanobu Kaneko for valuable discussions and comments. The first author is supported by the National Natural Science Foundation of China (under Grant 11671360), and the Fundamental Research Funds of Zhejiang Sci-Tech University (under Grant 2019Q063). The second author is supported by the China Scholarship Council (No. 201806310063).

%%%%%%%%%%%%%%%%%%%%%%%%%%%%%%%%%%%%%%%%%%%%%%%%%%%%%%
%%%%%%%%%%%%%%%%%%%%%%%%%%%%%%%%%%%%%%%%%%%%%%%%%%%%%%
%%%%%%%%%%%%%%%%%%%%%%%%%%%%%%%%%%%%%%%%%%%%%%%%%%%%%%
%%%%%%%%%%%%%%%%%%%%%%%%%%%%%%%%%%%%%%%%%%%%%%%%%%%%%%
%%%%%%%%%%%%%%%%%%%%%%%%%%%%%%%%%%%%%%%%%%%%%%%%%%%%%%

\end{document}